\def\mymode{i}  

\if\mymode p
\documentclass[10pt]{amsart}
\else
\documentclass[12pt]{amsart}
\usepackage{svninfo}
\fi

\usepackage{amsmath, amsthm, amsfonts, amssymb, latexsym, mathrsfs,
  stmaryrd}
\usepackage{hyperref}
\usepackage[utf8]{inputenc}

\if\mymode p
\usepackage{times,setspace}
\else
\renewcommand{\today}{%
  \number\day\space
  \ifcase\month\or
  January\or February\or March\or April\or May\or June\or
  July\or August\or September\or October\or November\or December\fi
  \space \number\year}
\fi

\newtheorem{thm}{Theorem}[section]
\newtheorem{prop}[thm]{Proposition}
\newtheorem{cor}[thm]{Corollary}
\newtheorem{lem}[thm]{Lemma}
\newtheorem{fact}[thm]{Fact}

\theoremstyle{definition}
\newtheorem{defn}[thm]{Definition}
\newtheorem{nota}[thm]{Notation}
\newtheorem{con}[thm]{Convention}

\theoremstyle{remark}
\newtheorem{rem}[thm]{Remark}

\DeclareSymbolFont{AMSb}{U}{msb}{m}{n}
\DeclareMathSymbol{\N}{\mathbin}{AMSb}{"4E}
\DeclareMathSymbol{\Z}{\mathbin}{AMSb}{"5A}
\DeclareMathSymbol{\R}{\mathbin}{AMSb}{"52}
\DeclareMathSymbol{\Q}{\mathbin}{AMSb}{"51}
\DeclareMathSymbol{\I}{\mathbin}{AMSb}{"49}
\DeclareMathSymbol{\C}{\mathbin}{AMSb}{"43}
\newcommand{\equals}{\mathrel{\mathop:}=}

\makeatletter

\def\dotminussym#1#2{%
  \setbox0=\hbox{$\m@th#1-$}%
  \kern.5\wd0%
  \hbox to 0pt{\hss\hbox{$\m@th#1-$}\hss}%
  \raise.8\ht0\hbox to 0pt{\hss$\m@th#1.$\hss}%
  \kern.5\wd0}
\newcommand{\dotminus}{\mathbin{\mathpalette\dotminussym{}}}

\makeatother

\title[A proof of completeness]{A proof of completeness for\\ continuous first-order logic}

\author{Itaï Ben Yaacov}
\address{Itaï \textsc{Ben Yaacov} \\
  Université Claude Bernard -- Lyon 1 \\
  Institut Camille Jordan \\
  43 boulevard du 11 novembre 1918 \\
  69622 Villeurbanne Cedex \\
  France}
\urladdr{\url{http://math.univ-lyon1.fr/~begnac}}

\author{Arthur Paul Pedersen}
\address{Arthur Paul Pedersen\\
Department of Philosophy\\ 
Carnegie Mellon University\\ 
Pittsburgh, PA 15213, USA} 
\email{\url{apaulpedersen@cmu.edu}}
\urladdr{\url{http://andrew.cmu.edu/~ppederse}}

\thanks{We wish to thank Jeremy Avigad for valuable comments. We also wish to thank Petr Hájek for offering useful remarks and for pointing us to important references.  Finally, we wish to thank an anonymous referee for helpful suggestions.}
\thanks{First author supported by
  ANR chaire d'excellence junior THEMODMET (ANR-06-CEXC-007) and
  by Marie Curie research network ModNet.}

\if\mymode i
\svnInfo $Id: CFOLCompleteness.tex 856 2009-03-16 09:10:11Z begnac $
\thanks{\textit{Revision} {\svnInfoRevision} \textit{of} \today}
\fi

\begin{document}

\begin{abstract} Continuous first-order logic has found interest among model theorists who wish to extend the classical analysis of ``algebraic'' structures (such as fields, group, and graphs) to various natural classes of complete metric structures (such as probability algebras, Hilbert spaces, and Banach spaces). With research in continuous first-order logic preoccupied with studying the model theory of this framework, we find a natural question calls for attention: Is there an interesting set of axioms yielding a completeness result?

The primary purpose of this article is to show that a certain, interesting set of axioms does indeed yield a completeness result for continuous first-order logic. In particular, we show that in continuous first-order logic a set of formulae is (completely) satisfiable if (and only if) it is consistent. From this result it follows that continuous first-order logic also satisfies an \emph{approximated} form of strong completeness, whereby $\Sigma\vDash\varphi$ (if and) only if $\Sigma\vdash\varphi\dotminus 2^{-n}$ for all $n<\omega$.  This approximated form of strong completeness asserts that if $\Sigma\vDash\varphi$, then proofs from $\Sigma$, being finite, can provide arbitrary better approximations of the truth of $\varphi$.

Additionally, we consider a different kind of question traditionally
arising in model theory -- that of decidability:
When is the set of all consequences of a theory (in a countable,
recursive language) recursive?
Say that a complete theory $T$ is \emph{decidable} if for every
sentence $\varphi$, the value $\varphi_T$ is a recursive real, and
moreover, uniformly computable from $\varphi$.
If $T$ is incomplete, we say it is decidable if for every sentence
$\varphi$ the real number $\varphi_{T}^{\footnotesize{\circ}}$ is
uniformly recursive from $\varphi$, where
$\varphi_{T}^{\footnotesize{\circ}}$ is the maximal value of $\varphi$
consistent with $T$.
As in classical first-order logic, it follows from the completeness
theorem of continuous first-order logic that if a complete theory
admits a recursive (or even recursively enumerable) axiomatization
then it is decidable.
\end{abstract}
\maketitle

\section{Introduction}
\label{sec:series}

Roughly speaking, model theory studies first-order theories and the
corresponding classes of their models (i.e., elementary classes). 
Properties of the first-order theory of a structure can then give
direct insight into the structure itself.  Investigation thereof was classically
restricted to so-called ``algebraic'' structures, such as fields,
groups, and graphs. Additionally, in modern model theory one often studies \emph{stable}
theories --- theories whose models admit a ``well-behaved'' notion of
independence (which, if it exists, is always unique).  

\emph{Continuous first-order logic} was developed  in
\cite{BenYaacov-Usvyatsov:CFO} as an extension of classical
first-order logic, permitting one to broaden the aforementioned
classical analysis of algebraic structures to various natural classes
of complete metric structures.
(It should be pointed out that  classes of complete metric structures
cannot be elementary in the classical sense for several reasons.
For example, 
completeness is an infinitary property and is therefore not 
expressible in classical first-order logic.
See also \cite{BenYaacov-Berenstein-Henson-Usvyatsov:NewtonMS} for a
general survey of continuous logic and its applications for various kinds of
metric structures arising in functional analysis and probability theory.)
For example, the class
(of unit balls) of Hilbert spaces and the class of probability algebras are elementary in this sense. (A
probability algebra is the Boolean algebra of events of a probability
space modulo the null measure ideal, with the metric $d(A,B)=\mu(A\Delta B)$.) Furthermore, the classical notion of
stability can easily be extended to continuous first-order logic. Indeed, somewhat unsurprisingly,
the classes of Hilbert spaces and probability algebras are stable,
independence being orthogonality and probabilistic independence,
respectively.  

 Historically, two groups of logics precede continuous first-order logic.   On the one hand, continuous first-order logic has \emph{structural} precursors.  The structural precursors are those logics which make use of machinery similar to that of continuous first-order logic yet were never developed to study complete metric structures.  Such structural precursors include Chang and Keisler's continuous logic \cite{Chang-Keisler:ContinuousModelTheory}, Łukasiewicz's many-valued logic \cite{Hajek:FuzzyLogic}, and Pavelka's many-valued logic \cite{Pavelka:OnFuzzyLogic}.  Chang and Keisler's logic is much too general for the study of complete metric structures, while Łukasiewicz logic and Pavelka's logic were developed for different purposes.
Nonetheless, continuous first-order logic is an improved variant of
Chang and Keisler's logic.
On the other hand, continuous first-order logic has \emph{purposive}
precursors.
The purposive precursors are those logics which were developed to
study complete metric structures yet do not make use of machinery
similar to that of continuous first-order logic.
The purposive precursors of continuous first-order logic include
Henson's logic for Banach structures \cite{Henson:NonstandardHulls}
and compact abstract theories (``cats'')
\cite{BenYaacov:PositiveModelTheoryAndCats,BenYaacov:SimplicityInCats,BenYaacov:Morley}.
Continuous first-order logic does not suffer from several shortcomings
of these logics.
Importantly, continuous first-order logic is less
technically involved than the previous logics and in many respects much
closer to classical first-order logic.
Still, continuous first-order logic is expressively equivalent to the
logic of metric open Hausdorff cats.
Continuous first-order logic also generalizes Henson's logic
for Banach structures, and as such, is expressively equivalent to a
natural variant of Henson's logic.

As an extension of classical first-order logic, continuous first-order logic satisfies suitably phrased forms of the compactness theorem, the L\"{o}wenheim-Skolem theorems, the diagram arguments, Craig's interpolation theorem, Beth's definability theorem, characterizations of quantifier elimination and model completeness, the existence of saturated and homogeneous models results, the omitting types theorem, fundamental results of stability theory, and nearly all other results of elementary model theory.
Moreover, continuous first-order logic affords a tractable framework for ultraproduct constructions (and so hull constructions) in applications of model theory in analysis and geometry.  In fact, expressing conditions from analysis and geometry feels quite natural in continuous first-order logic, furnishing model theorists and analysts with a common language.

Thus it is clear that continuous first-order logic is of interest to model theorists.  Yet with research focused on the model theory of continuous first-order logic and thus semantic features of this framework, a natural question seems to lurk in the background: Is there an interesting set of axioms yielding a completeness result? The answer depends on how one formulates the notion of completeness. To be sure, there is an interesting set of axioms, and,
as we will see, a set of formulae is (completely) satisfiable if (and only if) it is consistent.
However, as for continuous propositional logic, only an \emph{approximated} form of strong completeness is obtainable.  By this we mean that $\Sigma\vDash\varphi$ only if $\Sigma\vdash\varphi\dotminus2^{-n}$ for all $n<\omega$, which amounts to the idea if $\Sigma\vDash\varphi$, then proofs from $\Sigma$, being finite, can provide arbitrarily better approximations of the truth of $\varphi$. (What \emph{this} means will become clearer below.)  
This should hardly be surprising in light of the fact that continuous first-order logic has been developed for complete metric structures equipped with uniformly continuous functions with respect to which formulae take truth values anywhere in $[0,1]$.

Of course, a different kind of question traditionally arising in model theory is
that of decidability: When is the set of all consequences of a theory
(in a countable, recursive language) recursive?  Again, such questions
can be extended to continuous first-order logic.  Say that a complete theory $T$ is
\emph{decidable} if for every sentence $\varphi$, the value $\varphi_T$
is a recursive real, and moreover, uniformly computable from $\varphi$.
If $T$ is incomplete, we say it is decidable if for every sentence
$\varphi$ the real number $\varphi_{T}^{\footnotesize{\circ}}$ is uniformly recursive from
$\varphi$, where $\varphi_{T}^{\footnotesize{\circ}}$ is the maximal value of $\varphi$
consistent with $T$. (See Definition \ref{defn:deg}.  If $T$ is complete, then $\varphi$ takes the same value in all models of $T$, so $\varphi_{T}$ coincides with $\varphi_{T}^{\footnotesize{\circ}}$ and therefore $\varphi_{T}^{\footnotesize{\circ}}$ is the unique value of $\varphi$ consistent with $T$.)  As in classical first-order logic, it follows from the
completeness theorem that if a complete theory admits a recursive (or
even recursively enumerable) axiomatization then it is decidable, whence the connection
to the present paper.

Following an introduction to Łukasiewicz propositional logic and continuous propositional logic, we offer a definition of the language of continuous first-order logic and then supply a precise formulation of its semantics.   Indeed, this paper can also be seen as an effort to precisely organize and unify the various presentations of continuous first-order logic found in the literature which are often intimated in a rough-and-ready form.  Finally, we state and usually prove various results needed to reach the goal of this paper: to state and prove the completeness theorem for continuous first-order logic.  

To follow our intuitions to this end, the structure of our approach is largely borrowed from the classical approach employed to prove the completeness theorem.   In particular, we make use of a Henkin-like construction and a weakened version of the deduction theorem of classical first-order logic.  Moreover, various definitions and results found in the classical approach are translated to play analogous roles in our development, while from \cite{BenYaacov-Usvyatsov:CFO}  we take some basic facts and definitions peculiar to continuous first-order logic.  It will become apparent, however, that our approach differs from the classical one in many respects. Furthermore, the completeness theorem we offer is formulated with respect to the semantics employed by the model theorist who studies continuous first-order logic.  In particular, our work does not exploit an algebraic semantics. Finally, we should note that results of research on the interplay between logical and deductive entailment for both continuous propositional logic and Łukasiewicz propositional logic play a crucial role in getting our feet off the ground so that we may follow our intuitions in the first place \cite{BenYaacov:RandomVariables}. (This work of the first author is partially based on the results of work done by Chang \cite{Chang:ProofOfLukasiewiczAxiom, Chang:NewCompletenessProof} and Rose and Rosser \cite{Rose-Rosser:FragmentsManyValuedCalculi} on Łukasiewicz logic.)  With this, we set out to the task at hand.

\section{Łukasiewicz Logic and Continuous Logic}
\label{sec: luk and cont}
\begin{defn}
\label{def:luk}
Let $\mathcal{S}_{0}=\{P_{i}:i\in I\}$ be a set of distinct symbols.  Let $\mathcal{S}$ be freely generated from $\mathcal{S}_{0}$ by the formal binary operation $\dotminus$ and the unary operation $\neg$.  We call $\mathcal{S}$ a \emph{Łukasiewicz propositional logic}.  
\end{defn}
\begin{defn}\label{def:luksat}  Let $\mathcal{S}$ be a Łukasiewicz propositional logic.
\mbox{}
\begin{itemize}  
\item[(i)]   If $v_{0}:\mathcal{S}_{0}\to [0,1]$ is a mapping, we can extend $v_{0}$ to a unique mapping $v:\mathcal{S}\to [0,1]$ by setting
\begin{itemize}
\item[$\bullet$] $v(\varphi\dotminus\psi)\equals\max(v(\varphi)-v(\psi),0)$. 
\item[$\bullet$] $v(\neg\varphi)\equals1-v(\varphi)$. 
\end{itemize}
We call $v$ the \emph{truth assignment} defined by $v_{0}$.
\item[(ii)]  If $\Sigma\subseteq\mathcal{S}$, we write $v\vDash\Sigma$ if $v(\varphi)=0$ for all $\varphi\in\Sigma$, and we call $v$ a $model$ of $\Sigma$. We also write $v\vDash\varphi$ if $v\vDash\{\varphi\}$.
\item[(iii)] We say that $\Sigma\subseteq\mathcal{S}$ is \emph{satisfiable} if it has a model. 
\item[(iv)] We write $\Sigma\vDash\varphi$ if every model of $\Sigma$ is also a model of $\varphi$.
\end{itemize}
We may write $\Sigma\vDash^{\tiny\L}\!\varphi$ to indicate that we are dealing with a Łukasiewicz propositional logic.\end{defn}

\begin{rem}
Observe that 0 corresponds to truth and any $r\in(0,1]$ corresponds to a degree of truth or falsity, where 1 may be construed as absolute falsity.   Also observe that `$\dotminus$' plays a role analogous to that of `$\to$' in classical logic: We may interpret `$\psi\dotminus\varphi$' as `$\psi$ is \emph{implied} by $\varphi$,' `$\varphi$ is \emph{at least} as false as $\psi$', `$\psi$ is \emph{at most} as true as $\varphi$,' or simply, `$\psi$ is \emph{less than or equal} to $\varphi$.'  We prefer the last interpretation.
\end{rem}

\begin{defn}
\label{def:cont}
Let $\mathcal{S}_{0}=\{P_{i}:i\in I\}$ be a set of distinct symbols.  Let $\mathcal{S}$ be freely generated from $\mathcal{S}_{0}$ by the formal binary operation $\dotminus$ and the unary operations $\neg$ and $\frac12$.  We call $\mathcal{S}$ a \emph{continuous propositional logic}.  
\end{defn}
A truth assignment $v$ is defined as in (i) of Definition \ref{def:luksat} with the extra condition that $v(\frac12\varphi)\equals\frac12 v(\varphi)$.  Models, satisfiability, and logical entailment are defined as in Definition \ref{def:luksat}. We may write  $\Sigma\vDash^{\tiny\text{C}\L}\!\!\varphi$ to indicate that we are dealing with a continuous propositional logic.	  
\section{Axioms: Group 1}
\label{sec: Axioms}
We now present six of our fourteen axiom schemata.  The first four form an axiomatization for Łukasiewicz propositional logic \cite{Chang:ProofOfLukasiewiczAxiom, Rose-Rosser:FragmentsManyValuedCalculi}.\\  

  \newcounter{Lcount}
  \begin{list}{(A\arabic{Lcount})}
    {\usecounter{Lcount}
    \setlength{\rightmargin}{\leftmargin}}
  \item $(\varphi\dotminus\psi)\dotminus\varphi$
  \item $((\chi\dotminus\varphi)\dotminus(\chi\dotminus\psi))\dotminus (\psi\dotminus\varphi)$
  \item $(\varphi\dotminus(\varphi\dotminus\psi))\dotminus(\psi\dotminus(\psi\dotminus\varphi))$
  \item $(\varphi\dotminus\psi)\dotminus(\neg\psi\dotminus\neg\varphi)$\\
\end{list}

When the next two axiom schemata are added to the first four (in the appropriate language), we obtain an axiomatization for continuous propositional logic \cite{BenYaacov:RandomVariables}.\\
\begin{list}{(A\arabic{Lcount})}
{\usecounter{Lcount}}
\setcounter{Lcount}{4}
  \item $\frac12\varphi\dotminus(\varphi\dotminus\frac12\varphi)$
  \item $(\varphi\dotminus\frac12\varphi)\dotminus\frac12\varphi$\\  
\end{list}

Note that (A5) and (A6) say that $\frac{1}{2}$ behaves as it ought to.   Informally, under the intended interpretation, (A5) and (A6) taken together imply that $\frac12\varphi\dotplus\frac12\varphi=\varphi$. (`$\dotplus$' may be defined by setting $\varphi\dotplus\psi\equals\neg(\neg\varphi\dotminus\psi)$; thus, according to the intended interpretation of $`\dotminus$', the interpretation of $`\dotplus$' is given by $x\dotplus y=\min(x+y,1)$.) 

Formal deductions and the relation $\vdash$ for both logics are defined in the natural way, the only rule of inference being \emph{modus ponens}:
\begin{gather*}
  \frac{\varphi,\,\,\,\,\,\psi\dotminus\varphi}{\psi}
\end{gather*}
We can make this more precise as follows:  a \emph{formal deduction} from $\Sigma$ is a finite sequence of formulae $(\varphi_{i}:i<n)$ such that for each $i<n$, either (i) $\varphi_{i}$ is an instance of an axiom schema, (ii) $\varphi_{i}\in \Sigma$, or (iii) there are $j,k<i$ such that $\varphi_{k}=\varphi_{i}\dotminus\varphi_{j}$.  We accordingly say that $\varphi$ is \emph{provable from}  (or \emph{deducible from}, or \emph{a consequence of}) $\Sigma$  and write $\Sigma\vdash\varphi$ if there is a formal deduction from  $\Sigma$ ending in $\varphi$. Observe that a formula $\varphi$ is provable from $\Sigma$ just in case it is provable from a finite subset of $\Sigma$.  To avoid confusion, we may write $\Sigma\vdash^{\tiny\L}\!\varphi$ to indicate that $\varphi$ is provable from $\Sigma$ in Łukasiewicz propositional logic, and we may write $\Sigma\vdash^{\tiny\text{C}\L}\!\!\varphi$ to indicate that $\varphi$ is provable from $\Sigma$ in continuous propositional logic.  

It should be fairly clear to the reader how the proof systems of
continuous propositional logic and Łukasiewicz propositional logic
are related.
Both proof systems are of course sound, by which we mean that
if $\Sigma\vdash^{\tiny\L}\!\varphi$
(respectively, $\Sigma\vdash^{\tiny\text{C}\L}\!\!\varphi$),
then
$\Sigma\models^{\tiny\L}\!\varphi$
(respectively, $\Sigma\models^{\text{C\tiny\L}}\!\!\varphi$).
We leave this section with a definition, a few notational conventions, and a remark, each of which addresses some aspects of our presentation of continuous logic in this paper.
\begin{defn}  We say that a set of formulae $\Sigma$ is \emph{inconsistent} if $\Sigma\vdash\varphi$ for every formula $\varphi$ and \emph{consistent} otherwise. 
\end{defn}
\begin{nota}
Define $\psi\dotminus n\varphi$ by recursion on $n<\omega$: 
\mbox{}
\begin{itemize}
\item[(i)] $\psi\dotminus 0\varphi\equals\psi$.
\item[(ii)] $\psi\dotminus (n+1)\varphi\equals(\psi\dotminus n\varphi)\dotminus\varphi$.
\end{itemize}
\end{nota}
\begin{nota}
Let $1$ be shorthand for $\neg(\varphi\dotminus\varphi)$, where $\varphi$ is any formula, and let $2^{-n}$ be shorthand for $\underbrace{\frac12\cdots\frac12}_{n-times}1$.  Also, let $0$ be shorthand for $\neg 1$. 
\end{nota}

\begin{rem} Observe that on any truth assignment $v$ the set $\mathcal{D}$ generated from $1$ by applying the operations $\neg$, $\dotminus$, and $\frac{1}{2}$ is such that $\{v(d):d\in\mathcal{D}\}=\mathbb{D}$, the set of dyadic numbers, i.e., numbers of the form $\frac {k}{2^{n}}$, where $k,n<\omega$ and $k\leq 2^{n}$.  For simplicity of notation, we do not distinguish between the syntactic set $\mathcal{D}$ thus generated and $\mathbb{D}$.
\end{rem}

\section{Black Box Theorems}\label{sec:bbt}  We now record several results which will be used in this paper.

\begin{fact}[Weak Completeness for Łukasiewicz Logic \cite{Chang:NewCompletenessProof, Rose-Rosser:FragmentsManyValuedCalculi}]
\label{fact:weak}  Let $\mathcal{S}$ be a Łukasiewicz propositional logic and $\varphi\in\mathcal{S}$.  Then $\vDash\varphi$ if and only if $\,\vdash\varphi$. 
\end{fact}

\begin{fact}\label{fact:lukcomp} Let $\mathcal{S}$ be a  Łukasiewicz propositional logic, and $\Sigma\subseteq\mathcal{S}$.  Then $\Sigma$ is consistent if and only if it is satisfiable.
\end{fact}
\begin{fact}\label{fact:contcomp} Let $\mathcal{S}$ be a continuous propositional logic and $\Sigma\subseteq\mathcal{S}$.  Then $\Sigma$ is consistent if and only if it is satisfiable. 
\end{fact}
An immediate corollary of the previous fact is the following approximated form of strong completeness.
\begin{fact}[Approximated Strong Completeness for Continuous Logic]\label{fact:approx}
Let $\mathcal{S}$ be a continuous propositional logic, $\Sigma\subseteq\mathcal{S}$, and $\varphi\in\mathcal{S}$.  Then $\Sigma\vDash\varphi$ if and only if $\Sigma\vdash\varphi\dotminus 2^{-n}$ for all $n<\omega$.
\end{fact}
In fact, this is the best we can hope for.  To see why, consider $\Sigma\equals\{P\dotminus 2^{-n}:n<\omega\}$. Then $\Sigma\vDash P$, yet for no finite $\Sigma_{0}\subseteq\Sigma$ do we have $\Sigma_{0}\vDash P$.  However, we do have the following weaker result.
\begin{fact}[Finite Strong Completeness for Continuous Logic] \label{fact:finite}
Let $\mathcal{S}$ be a continuous propositional logic, $\Sigma\subseteq\mathcal{S}$ be finite, and $\varphi\in\mathcal{S}$. Then $\Sigma\vDash\varphi$ if and only if $\Sigma\vdash\varphi$.
\end{fact}

In view of the above facts, observe that if $\mathcal{S}$ is a Łukasiewicz or continuous propositional logic and $\Sigma\subseteq \mathcal{S}$,
\begin{itemize}
\item $\Sigma$ is inconsistent if and only if $\Sigma\vdash^{\tiny\L} 1$.
 \item $\Sigma$ is inconsistent if and only if $\Sigma\vdash^{\tiny\text{C}\L}\!\!\ d$ for some $d\in\mathbb{D}\backslash\{0\}$.
 \end{itemize}

 Fact \ref{fact:lukcomp} -- Fact \ref{fact:finite} have been established independently by the first author in \cite{BenYaacov:RandomVariables}. Nonetheless, Fact \ref{fact:lukcomp} has been proved in \cite{BelluceChang:1963} (indeed, for Łukasiewicz first-order logic; see also \cite{Hajek:FuzzyLogic}) and Fact \ref{fact:approx} has a counterpart in \cite{Hay:1963} (again, see also \cite{Hajek:FuzzyLogic}), while  Fact \ref{fact:finite} has been proved for rational Pavelka propositional logic \cite{Hajek:FuzzyLogic}  and for Łukasiewicz propositional logic \cite{CDM:2000,Hajek:FuzzyLogic}.

\section{The Language and Semantics of Continuous First-Order Logic}
\label{sec:sem}

\begin{defn} The \emph{logical symbols} of continuous first-order logic are
\mbox{}
\begin{itemize}
\item Parentheses: ( , )
\item Connectives: $\dotminus\, ,\, \neg\,,\, \frac12$
\item Quantifiers: $\sup$ $(\text{or}\, \inf)$
\item Variables: $v_{0},v_{1},\ldots$
\item An optional binary metric: $d$
\end{itemize}
\end{defn}
The $\sup$-quantifier plays the role of the $\forall$-quantifier from classical first-order logic, whereas the $\inf$-quantifier plays the role of the $\exists$-quantifier from classical first-order logic.  In fact, the $\sup$-quantifier can be defined in terms of the $\inf$-quantifier. Furthermore, instead of an optional binary congruence relation symbol $\approx$ as in classical first-order logic,  continuous first-order logic has an optional binary metric symbol $d$. 

We should now say something about the choice of connectives.  In classical first-order logic, a set of connectives is \emph{complete} if every mapping  from $\{0, 1\}^{n}\to\{0, 1\}$ can be written using the
set of connectives.  The set of connectives $\{\neg,\to\}$ is complete in this sense.  Similarly, in continuous first-order logic, a set of connectives is \emph{complete} if every  continuous mapping from $[0, 1]^{n} \to [0, 1]$ can be written using the set of connectives.   But this notion is much too demanding.  Indeed, a complete set of connectives would require continuum many connectives.  A more reasonable demand of a set of connectives is that every continuous function from $[0, 1]^{n} \to [0, 1]$ can be written using the set of connectives \emph{up to} arbitrarily better approximations.   A set of connectives satisfying this requirement is said to be \emph{full} (see \cite{BenYaacov-Usvyatsov:CFO}; see also \cite{CDM:2000}).  The set of connectives $\{\dotminus,\neg,\frac{1}{2}\}$ is full in this sense, and  $\{\dotminus,\frac{1}{2},0,1\}$ is full as well --- so is any set of connectives which includes $\dotminus$, $1$, and some dense set $\mathcal{C}\subseteq [0,1]$, such as $\mathbb{D}$ or $\Q\cap[0,1]$.  Here we choose to canonize $\{\dotminus,\neg,\frac{1}{2}\}$ as the set connectives.  As it is finite, this set of connectives is an economical choice,  and it is the analogue of the popular choice of connectives $\{\neg,\to\}$ of classical first-order logic.
\begin{defn}
A \emph{continuous signature} is a quadruple $\mathcal{L}=(\mathcal{R},\mathcal{F}, \mathcal{G},n)$ such that
\mbox{}
\begin{itemize}
\item[(i)] $\mathcal{R}\neq\emptyset$.
\item[(ii)] $\mathcal{R}\cap\mathcal{F}=\emptyset$.
\item[(iii)] $\mathcal{G}$ has the form $\{\delta_{s,i}:(0,1]\to (0,1]: s\in\mathcal{R}\cup\mathcal{F}\text{ and }i<n_{s}\}$.
\item[(iv)] $n:\mathcal{R}\cup\mathcal{F}\to\omega$.  
\end{itemize}
Members of $\mathcal{R}$ are called \emph{relation symbols}, while members of $\mathcal{F}$ are called \emph{function symbols}.  For each $s\in\mathcal{R}\cup\mathcal{F}$, we write $n_{s}$ for the value of $s$ under $n$ and call $n_{s}$ the $arity$ of $s$.  We call a member of $\mathcal{G}$ a \emph{modulus of  uniform continuity}.
\end{defn}

\begin{rem}Observe that members of $\mathcal{G}$ are not syntactic objects. Rather, each member of $\mathcal{G}$ is a genuine operation on $(0,1]$.  The role of the moduli of uniform continuity will be clarified below.
\end{rem}

\begin{defn} A \emph{continuous signature with a metric} is a continuous signature with a distinguished binary relation symbol $d$.\end{defn}

Terms, formulae, and notions of free and bound substitution are defined in the usual way (see $\mathsection\ref{sec:sub}$ concerning these notions).  We denote the set of variables by $V$ and the set of formulae by $\textup{For}(\mathcal{L})$. In classical first-order logic, a metric symbol $d$ would most naturally be thought of as a function symbol.  In continuous first-order logic, however, $d$ is a \emph{relation} symbol, and as such, $dt_{0}t_{1}$ is a formula rather than a term.  

 We write `$\sup_{x}$' (and `$\inf_{x}$') instead of `$\sup x$' (and `$\inf x$').  We define `$\inf$,' `$\wedge$,' `$\vee$,' and `$|\, |$' by setting
 \begin{eqnarray*}
 \inf\!\!\,_{x}\,\varphi&\equals&\neg\sup\!\!\,_{x}\,\neg\varphi.\\
 \varphi\wedge\psi&\equals&\varphi\dotminus(\varphi\dotminus\psi).\\
  \varphi\vee\psi&\equals&\neg(\neg\varphi\wedge\neg\psi).\\
|\varphi-\psi|&\equals&(\varphi\dotminus\psi)\vee(\psi\dotminus\varphi).
 \end{eqnarray*}
The reader should verify that the definition of `$\inf$' accords with the obvious intended interpretation when the semantics are given below.   Observe that (A2) becomes `$(\varphi\wedge\psi)\dotminus(\psi\wedge\varphi)$.'  Also, according to the intended interpretation of `$\dotminus$,' the interpretation of `$\wedge$' is given by $x\wedge y=\min(x,y)$. Thus, according to the interpretation of `$\wedge$,' the interpretation of `$\vee$' is given by $x\vee y=\max(x,y)$, whereby the interpretation of `$|\,| $' is one-dimensional Euclidean distance. As 0 corresponds to truth, it should be clear that the interpretation of `$\wedge$' is the continuous analogue of the classical `or,' whereas the interpretation of `$\vee$' is the continuous analogue of the classical `and.' 
\begin{defn}
\mbox{}
\begin{itemize} 
\item[(i)] Let $(M,d)$ and $(M',d')$ be metric spaces and $f:M\to M'$ be a function.  We say that $\delta :(0,1]\to (0,1]$ is a \emph{modulus of uniform continuity for f} if for each $\epsilon\in(0,1]$ and all $a,b\in M$, $d(a,b)<\delta(\epsilon)\Longrightarrow d^{\prime}(f(a),f(b))\leq\epsilon$.
\item[(ii)] We say that $f$ is \emph{uniformly continuous} if there is a modulus of uniform continuity for $f$.
\end{itemize}
\end{defn}

\begin{defn} 
\label{defn:structure}Let $\mathcal{L}$ be a continuous signature.  A \emph{continuous} $\mathcal{L}$-\emph{pre}-\emph{structure} is an ordered pair $\mathfrak{M}=(M,\rho)$, where $M$ is a non-empty set and $\rho$ is a function whose domain is $\mathcal{R}\cup\mathcal{F}$, such that:
\mbox{}
\begin{itemize}   
\item[(i)] To each $f\in\mathcal{F}$, $\rho$ assigns a mapping $f^{\mathfrak{M}}: M^{n_{f}}\to M$.
\item[(ii)] To each $P\in\mathcal{R}$, $\rho$ assigns a mapping $P^{\mathfrak{M}}: M^{n_{P}}\to [0,1]$.
\end{itemize}
If $\mathcal{L}$ is also a continuous signature with a metric, then $\mathfrak{M}$ is such that:
\begin{itemize}
\item[(iii)] $\rho$ assigns to $d$ a pseudo-metric $d^{\mathfrak{M}}: M\times M\to [0,1]$.   
\item[(iv)] For each $f\in\mathcal{F}$, $i<n_{f}$, and  $\epsilon\in(0,1]$, $\delta_{f,i}$ satisfies the following condition:\\
$\forall \bar{a},\bar{b},c,e\,\,[d^{\mathfrak{M}}(c,e)<\delta_{f,i}(\epsilon)\Longrightarrow d^{\mathfrak{M}}(f^{\mathfrak{M}}(\bar{a},c,\bar{b}),f^{\mathfrak{M}}(\bar{a},e,\bar{b}))\leq\epsilon]$,\\ where $|\bar{a}|=i$ and $|\bar{a}| +|\bar{b}| +\, 1=n_{f}$; we thereby call $\delta_{f,i}$ the \emph{uniform continuity modulus of f with respect to the }$i$\text{th }\emph{argument}.

\item[(v)] For each $P\in\mathcal{R}$, $i<n_{P}$, and $\epsilon\in(0,1]$, $\delta_{P,i}$ satisfies the following condition:\\
$\forall \bar{a},\bar{b},c,e\,\,[d^{\mathfrak{M}}(c,e)<\delta_{P,i}(\epsilon)\Longrightarrow\! \max(P^{\mathfrak{M}}(\bar{a},c,\bar{b}) - P^{\mathfrak{M}}(\bar{a},e,\bar{b}), 0) \leq\epsilon]$,\\ where $|\bar{a}|=i$ and $|\bar{a}| +|\bar{b}| +\, 1=n_{P}$; we thereby call $\delta_{P,i}$ the \emph{uniform continuity modulus of P with respect to the }$i$\text{th }\emph{argument}. 
\end{itemize}   
\end{defn}
Here, as well as in the remainder of this paper, a tuple $a_{0},\ldots,a_{n-1}$ may be for convenience denoted by $\bar{a}$, whereby $|\bar{a}|$ will denote the length of $\bar{a}$.

\begin{con} In this paper we make the convention that for an $n$-ary function $f$, the $i$th argument of the expression $f(x_{0},\ldots,x_{n-1})$ is $x_{i}$.  In particular, the $\textit{0}\,$th argument of $f$ is $x_{0}$. 
\end{con}

\begin{rem}\label{rem:congruence} Conditions (iv) and (v) correspond to the congruence axioms of classical first-order logic. In classical first-order logic it is required that a distinguished binary relation symbol $\approx$ be such that for each classical structure $\mathfrak{A}$, the relation $\approx^{\mathfrak{A}}$ is an equivalence relation satisfying the following congruence axioms:
\mbox{}
\begin{itemize}\item For each function symbol $f$ and $i<n_{f}$, \\ $\forall \bar{a},\bar{b},c,e\,\,
(c\approx^{\mathfrak{A}} e\Longrightarrow f^{\mathfrak{A}}(\bar{a},c,\bar{b})\approx^{\mathfrak{A}}f^{\mathfrak{A}}(\bar{a},e,\bar{b}))$,\\
where $|\bar{a}|=i$ and $|\bar{a}| +|\bar{b}| +\, 1=n_{f}$.
\item For each relation symbol $P$ and $i<n_{P}$, \\ $\forall \bar{a},\bar{b},c,e\,\,
(c\approx^{\mathfrak{A}} e\Longrightarrow(P^{\mathfrak{A}}(\bar{a},c,\bar{b})\Rightarrow P^{\mathfrak{A}}(\bar{a},e,\bar{b})))$,\\
where $|\bar{a}|=i$ and $|\bar{a}| +|\bar{b}| +\, 1=n_{P}$.
\end{itemize} 
 In classical first-order logic, one can thereby show that for each classical structure $\mathfrak{A}$, the relation $\approx^{\mathfrak{A}}$ satisfies the following properties:
 \begin{itemize}\item For each function symbol $f$, \\ $\forall \bar{a},\bar{b}\,\,
((\bigwedge_{i<n_{f}}a_{i}\approx^{\mathfrak{A}} b_{i})\Longrightarrow f^{\mathfrak{A}}(\bar{a})\approx^{\mathfrak{A}}f^{\mathfrak{A}}(\bar{b}))$,\\ 
where $|\bar{a}|=|\bar{b}|=n_{f}$.
\item For each relation symbol $P$, \\ $\forall \bar{a},\bar{b}\,\,
((\bigwedge_{i<n_{P}}a_{i}\approx^{\mathfrak{A}} b_{i})\Longrightarrow(P^{\mathfrak{A}}(\bar{a})\Rightarrow P^{\mathfrak{A}}(\bar{b})))$,\\
where $|\bar{a}|=|\bar{b}|=n_{P}$.
\end{itemize} 
 Along a similar vein, in continuous first-order logic one can show that for each function symbol $f$ and predicate symbol $P$, there are moduli of uniform continuity $\Delta_{f}$ and $\Delta_{P}$ which depend on their uniform continuity moduli $(\delta_{f,i}:i<n_{f})$ and $(\delta_{P,i}:i<n_{P})$, respectively.  For example, $\Delta_{f}:(0,1]\to(0,1]$ may be defined by setting $\Delta_{f}(\epsilon)\equals\min\{\delta_{f,0}(\epsilon/n_{f}),\ldots,\delta_{f,n_{f}-1}(\epsilon/n_{f})\}$.  Accordingly, in each continuous $\mathcal{L}$-pre-structure $\mathfrak{M}$, the mapping $\Delta_{f}$ is a modulus of uniform continuity for $f^{\mathfrak{M}}$ with respect to the maximum metric $D^\mathfrak{M}_{n_{f}}$ defined by 
 $D^\mathfrak{M}_{n_{f}}(\bar{a},\bar{b})\equals\max_{i<n_{f}}(d^{\mathfrak{M}}(a_{i},b_{i}))=\bigvee_{i<n_{f}}d^{\mathfrak{M}}(a_{i},b_{i})$.  This may be expressed formally as follows:
 \begin{itemize}
 \item For each $f\in\mathcal{F}$ and  $\epsilon\in(0,1]$, $\Delta_{f}$ satisfies the following condition:\\
$\forall \bar{a},\bar{b}\,\,[(\bigvee_{i<n_{f}}d^{\mathfrak{M}}(a_{i},b_{i}))<\Delta_{f}(\epsilon)\Longrightarrow d^{\mathfrak{M}}(f^{\mathfrak{M}}(\bar{a}),f^{\mathfrak{M}}(\bar{b}))\leq\epsilon]$,\\
where $|\bar{a}|=|\bar{b}|=n_{f}$.
\end{itemize}
In light of the foregoing discussion, observe that for every continuous $\mathcal{L}$-pre-structure $\mathfrak{M}$ and $s\in \mathcal{R}\cup\mathcal{F}$, $s^{\mathfrak{M}}$ is a uniformly continuous function.
\end{rem}

 \begin{defn}
A \emph{continuous} $\mathcal{L}$-\emph{structure} is a continuous $\mathcal{L}$-pre-structure $\mathfrak{M}=(M,\rho)$ such that $\rho$ assigns to $d$ a \emph{complete} metric: 
\mbox{}
\begin{itemize}
\item[(i)] $\forall a,b \,\,(d^{\mathfrak{M}}(a,b)=0\Longrightarrow a=b)$.
\item[(ii)] Every Cauchy sequence converges. 
\end{itemize}
\end{defn}

\begin{defn} 
\mbox{}
\begin{itemize} 
\item[(i)] If $\mathfrak{M}$ is a continuous $\mathcal{L}$-pre-structure, then an $\mathfrak{M}$-\emph{assignment} is a mapping $\sigma: V\to M$.
\item[(ii)] If $\sigma$ is an $\mathfrak{M}$-assignment, $x\in V$, and $a\in M$, we define an $\mathfrak{M}$-assignment $\sigma^{a}_{x}$ by setting for all $y\in V$,\begin{eqnarray*}
\sigma^{a}_{x}(y)&\equals&\begin{cases} a & \text{if}\,\,\,\, x=y\\
 \sigma(y)& \text{otherwise}.
 \end{cases}
\end{eqnarray*}
\end{itemize}   
\end{defn}
The interpretation of a term $t$ in a continuous $\mathcal{L}$-pre-structure $\mathfrak{M}$ is defined as in classical first-order logic.  We denote its interpretation by $t^{\mathfrak{M},\sigma}$.

\begin{defn} Let $\mathfrak{M}$ be a continuous $\mathcal{L}$-pre-structure.  For a formula $\varphi$ and an $\mathfrak{M}$-assignment $\sigma$, we define \emph{the value of} $\varphi$ \emph{in} $\mathfrak{M}$ \emph{under} $\sigma$, $\mathfrak{M}(\varphi,\sigma)$, by induction:
\begin{itemize}
\item[(i)] $\mathfrak{M}(Pt_{0}\cdots t_{n_{P}-1},\sigma)\equals P^{\mathfrak{M}}(t^{\mathfrak{M},\sigma}_{0},\ldots,t^{\mathfrak{M},\sigma}_{n_{P}-1})$.
\item[(ii)] $\mathfrak{M}(\alpha\dotminus\beta,\sigma)\equals\max(\mathfrak{M}(\alpha,\sigma)-\mathfrak{M}(\beta,\sigma), 0)$.
\item[(iii)] $\mathfrak{M}(\neg\alpha,\sigma)\equals1-\mathfrak{M}(\alpha,\sigma)$.
\item[(iv)] $\mathfrak{M}(\frac12\alpha,\sigma)\equals\frac12\mathfrak{M}(\alpha,\sigma)$.
\item[(v)] $\mathfrak{M}(\sup_{x}\alpha,\sigma)\equals\sup\{\mathfrak{M}(\alpha,\sigma^{a}_{x}): a\in M\}$.
\end{itemize}
\end{defn}
If one chooses to use `$\inf$' instead of `$\sup$,' one may replace (v) with $\mathfrak{M}(\inf_{x}\alpha,\sigma)\equals\inf\{\mathfrak{M}(\alpha,\sigma^{a}_{x}): a\in M\}$.

\begin{defn} Let $\mathfrak{M}$ be a continuous $\mathcal{L}$-pre-structure, let $\sigma$ be an $\mathfrak{M}$-assignment, and let $\Gamma\subseteq\textup{For}(\mathcal{L})$.
\mbox{}
\begin{itemize} 
\item[(i)] We say that $(\mathfrak{M},\sigma)$ \emph{models} (or \emph{satisfies}) $\Gamma$ and that $(\mathfrak{M},\sigma)$ is a \emph{model} of $\Gamma$, written $(\mathfrak{M},\sigma)\vDash_{Q}\Gamma$, if $\mathfrak{M}(\varphi,\sigma)=0$ for all $\varphi\in\Gamma$.  We of course say that $(\mathfrak{M},\sigma)$ is a \emph{model} of $\varphi$ and write $(\mathfrak{M},\sigma)\vDash_{Q}\varphi$ if $(\mathfrak{M},\sigma)\vDash_{Q}\{\varphi\}$.   
\item[(ii)] We say that $\Gamma$ is \emph{satisfiable} if it has a model.
\end{itemize}
\end{defn}

\begin{defn}
Let $\Gamma$ be a set of formulae and $\varphi$ be a formula.  We write $\Gamma\vDash_{Q}\varphi$ if every model of $\Gamma$ is a model of $\varphi$.   If $\emptyset\vDash_{Q}\varphi$, we say that $\varphi$ is \emph{valid}.
\end{defn}
\begin{defn}
We write $\varphi\equiv\psi$ if $\mathfrak{M}(\varphi,\sigma)=\mathfrak{M}(\psi,\sigma)$ for every continuous $\mathcal{L}$-pre-structure and $\mathfrak{M}$-assignment $\sigma$.
\end{defn} 
 With these definitions, many properties analogous to those of classical first-order logic can be derived (see \cite{Ebbinghaus-Flum-Thomas:MathematicalLogic}).

\section{Substitution and Metric Completions}
\label{sec:sub}
\subsection*{Substitution}As in classical fist-order logic, \emph{free} and \emph{bound} substitution play an important role in connecting the syntax with the semantics.  This brief subsection is intended to remind the reader of these two notions of substitution and to indicate to the reader what features of these notions are crucial for our purposes.
\begin{defn}
Let $t$ be a term, and let $x$ be a variable.  We define the \emph{free substitution} of $t$ for $x$ inside a formula $\varphi$, $\varphi[t/x]$, as the result of replacing $x$ by $t$ in $\varphi$ if $x$ occurs free in $\varphi$. We say that $\varphi[t/x]$ is \emph{correct} if no variable $y$ in $t$ is captured by a $\sup_{y}$ (or $\inf_{y}$) quantifier in $\varphi[t/x]$. 
\end{defn}

The next lemma is the continuous analogue of the substitution lemma of classical first-order logic.

\begin{lem}[Substitution Lemma]\label{lem:inter}Let $\mathfrak{M}$ be a continuous $\mathcal{L}$-pre-structure, and let $\sigma$ be an $\mathfrak{M}$-assignment.  Let $t$ be a term, $x$ be a variable, and $\varphi$ be a formula.  Put $a\equals t^{\mathfrak{M},\sigma}$.  Suppose $\varphi[t/x]$ is correct.  Then 
\begin{eqnarray*}
\mathfrak{M}(\varphi[t/x],\sigma)&=&\mathfrak{M}(\varphi,\sigma^{a}_{x}).
\end{eqnarray*}
\end{lem}

\begin{defn}  Let $\varphi$ be a formula, and let $x,y$ be variables.  We define the \emph{bound substitution} of $y$ for $x$ inside $\varphi$, $\varphi\{y/x\}$, as the result of replacing each subformula $\sup_{x}\alpha$ (or $\inf_{x}\alpha$) of $\varphi$ by $\sup_{y}\alpha[y/x]$ (or $\inf_{y}\alpha$).  We say that  $\varphi\{y/x\}$ is \emph{correct} if $y$ is not free in $\alpha$ and $\alpha[y/x]$ is correct  in each such subformula $\sup_{x}\alpha$ (or $\inf_{x}\alpha$).
\end{defn}

The following lemma is an immediate result. The reader interested may consult \cite{Ebbinghaus-Flum-Thomas:MathematicalLogic} for a classical first-order proof to see how a proof for continuous first-order logic may be constructed.

\begin{lem}[Bound Substitution Lemma]\label{lem:bdsub}
Let $\varphi$ be a formula, and let $x_{0},\ldots,x_{n-1}$ be a finite sequence of variables.  Then by a sequence of bound substitutions there is a formula $\varphi'$ in which $x_{0},\ldots,x_{n-1}$ are not bound and $\varphi\equiv\varphi'$.
\end{lem} 
\subsection*{Metric Completions}
We will presently see that each continuous $\mathcal{L}$-pre-structure (for a continuous signature with a metric) is virtually indistinguishable from its metric completion (Theorem \ref{thm:completion}).  To this end, we first offer a definition.

\begin{defn} Let $\mathfrak{M}$ and $\mathfrak{N}$ be continuous $\mathcal{L}$-pre-structures, and let $h:M\to N$.  We call $h$ an $\mathcal{L}$-\emph{morphism} of $\mathfrak{M}$ into $\mathfrak{N}$ if $h$ satisfies the following two conditions:
\begin{itemize}
\item[(i)]  For each $f\in\mathcal{F}$ and all $a_{0},\ldots,a_{n_{f}-1}\in M$,
\begin{eqnarray*}
h(f^{\mathfrak{M}}(a_{0},\ldots,a_{n_{f}-1}))&=&f^{\mathfrak{N}}(h(a_{0}),\ldots,h(a_{n_{f}-1})).
\end{eqnarray*}
\item[(ii)] For each $P\in\mathcal{R}$ and all $a_{0},\ldots,a_{n_{P}-1}\in M$,
\begin{eqnarray*}
P^{\mathfrak{M}}(a_{0},\ldots,a_{n_{P}-1})&=&P^{\mathfrak{N}}(h(a_{0}),\ldots,h(a_{n_{P}-1})).
\end{eqnarray*}
\end{itemize}

We call $h$ an \emph{elementary} $\mathcal{L}$-\emph{morphism} if $h$ is an $\mathcal{L}$-morphism and for every $\mathfrak{M}$-assignment $\sigma$ and formula $\varphi$, $\mathfrak{M}(\varphi,\sigma)=\mathfrak{N}(\varphi,h\circ\sigma)$.
\end{defn}

Observe that if $\mathcal{L}$ is a continuous signature with a metric $d$ and $h$ is an $\mathcal{L}$-{morphism} of $\mathfrak{M}$ into $\mathfrak{N}$, then for all $a_{0},a_{1}\in M$,
\begin{eqnarray*}
d^{\mathfrak{M}}(a_{0},a_{1})&=&d^{\mathfrak{N}}(h(a_{0}),h(a_{1})).
\end{eqnarray*}
In other words, $h$ is an isometry.

 An immediate consequence of our definitions is the following theorem.

\begin{thm}\label{thm:homo} If $\mathfrak{M}$ and $\mathfrak{N}$ are continuous $\mathcal{L}$-pre-structures and $h:M\to N$ is a surjective $\mathcal{L}$-morphism, then $h$ is an elementary $\mathcal{L}$-morphism of $\mathfrak{M}$ onto $\mathfrak{N}$.
\end{thm}
\begin{proof}  Straightforward induction on formulae.
\end{proof}

\begin{defn} Let $\mathfrak{M}$ be a continuous $\mathcal{L}$-pre-structure.  Let $t$ be a term and $\varphi$ be a formula such that all variables occurring in $t$ and all free variables occurring in $\varphi$ appear among $n$ distinct variables $x_{0},\ldots,x_{n-1}$.
 Define functions $\widetilde{t}^{\,\mathfrak{M},\bar{x}}:M^{n}\to M$ and $\widetilde{\varphi}^{\,\mathfrak{M},\bar{x}}:M^{n}\to [0,1]$ by setting for all $\bar{a}\in M^{n}$, 
\begin{eqnarray*}
t^{\mathfrak{M},\bar{a}}&\equals& t^{\mathfrak{M},\sigma}\\
\widetilde{\varphi}^{\,\mathfrak{M},\bar{x}}(\bar{a})&\equals&\mathfrak{M}(\varphi,\sigma),
\end{eqnarray*}
where $\sigma$ is an $\mathfrak{M}$-assignment such that $\sigma(x_{i})=a_{i}$ for each $i<n$.
\end{defn}

 We have all of the ingredients necessary to prove the following theorem.

\begin{thm}
\label{thm:mod}
Let $\mathfrak{M}$ be a continuous $\mathcal{L}$-pre-structure. Then for every term $t$ and formula $\varphi$  the mappings $\widetilde{t}^{\,\mathfrak{M},\bar{x}}:M^{n}\to M$ and $\widetilde{\varphi}^{\,\mathfrak{M},\bar{x}}:M^{n}\to [0,1]$ are uniformly continuous.
\end{thm}
\begin{proof} By induction on terms and formulae, using the fact that moduli of uniform continuity can be built up from other moduli, as mentioned in Remark \ref{rem:congruence}.
\end{proof}
Using Lemma \ref{lem:bdsub}, Theorem \ref{thm:homo}, and Theorem \ref{thm:mod} , one can prove the following: 

\begin{thm}[Existence of Metric Completion]\label{thm:completion} Let $\mathcal{L}$ be a continuous signature with a metric, and let $\mathfrak{M}$ be a continuous $\mathcal{L}$-pre-structure.  Then there is a continuous $\mathcal{L}$-structure $\mathfrak{\widehat{M}}$ and an elementary $\mathcal{L}$-morphism of $\mathfrak{M}$ into $\mathfrak{\widehat{M}}$.
\end{thm}

\begin{proof} The proof invokes elementary facts about metric spaces, pseudo-metrics, Cauchy sequences, and metric completions.  Essential to this proof is the fact that for all metric spaces
$(M,d)$ and $(M',d')$ such that $(M',d')$ is complete and $N\subseteq M$, if $f:N\to M'$ is a mapping and $\delta$ is a modulus of uniform continuity for $f$, then $f$ can be uniquely extended to a function $\bar{f}: \bar{N}\to M'$ such that $\delta$ is a modulus of uniform continuity for $\bar{f}$ (where $\bar{N}$ denotes the closure of $N$ in $M$).  This fact is important insofar as it implies that an underlying continuous signature with a metric will not have to be altered for a metric completion.
\end{proof}

\begin{defn} Let $\mathcal{L}$ be a continuous signature with a metric, let $\Gamma\subseteq \textup{For}(\mathcal{L})$, and let $\varphi$ be a formula.
\mbox{}
\begin{itemize}
\item[(i)]We write $\Gamma\vDash_{QC}\varphi$ if for every continuous $\mathcal{L}$-structure $\mathfrak{M}$ and $\mathfrak{M}$-assignment $\sigma$, if $(\mathfrak{M},\sigma)\vDash_{Q}\Gamma$, then $(\mathfrak{M},\sigma)\vDash_{Q}\varphi$.
\item[(ii)]  We say that $\Gamma$ is \emph{completely satisfiable} or that $\Gamma$ has a \emph{complete model} if there is a continuous $\mathcal{L}$-structure $\mathfrak{M}$ and $\mathfrak{M}$-assignment $\sigma$ such that $(\mathfrak{M},\sigma)$ models $\Gamma$.
\end{itemize}
\end{defn}

We therefore have the following corollary:
\begin{cor}
Let $\mathcal{L}$ be a continuous signature with a metric, and let $\Gamma\subseteq \textup{For}(\mathcal{L})$.  Then for every formula $\varphi$,
\mbox{}
\begin{itemize}
\item[(i)] $\Gamma\vDash_{Q}\varphi$ if and only if $\Gamma\vDash_{QC}\varphi$.
\item[(ii)] $\Gamma$ is satisfiable if and only if $\Gamma$ is \emph{completely satisfiable}.
\end{itemize}
\end{cor}
\noindent We may thereby restrict the notion of logical entailment to continuous $\mathcal{L}$-structures.    

In classical first-order logic, if one were to require a distinguished binary relation symbol $\approx$ be interpreted only as a congruence relation in each classical structure rather than as strict equality, then one could show for each such structure that there is a surjective elementary $\mathcal{L}$-morphism (in the obvious sense) onto its quotient structure.    Accordingly, we would find that classical first-order logic does not distinguish between structures which require that $\approx$ be interpreted as a congruence relation and structures which require that $\approx$ be interpreted as strict equality.  

 In continuous first-order logic, we see a somewhat analogous result.  We require that a distinguished binary relation symbol $d$ be interpreted as a pseudo-metric that satisfies congruence properties, akin to those of classical first-order logic (cf. Definition \ref{defn:structure}, parts (iii) and (iv)).  One can show that for each $\mathcal{L}$-pre-structure, there is a surjective elementary $\mathcal{L}$-morphism onto its quotient structure (which essentially transforms $d$ into a genuine metric).  Yet one can go one step further:  For any  $\mathcal{L}$-pre-structure with a genuine metric, there is an injective elementary $\mathcal{L}$-morphism (and \emph{a fortiori} an isometry) into its completion, a continuous $\mathcal{L}$-structure.  Therefore, continuous first-order logic does not distinguish between continuous $\mathcal{L}$-pre-structures and continuous $\mathcal{L}$-structures.  Theorem \ref{thm:completion} can be seen as encoding this fact.  As you would expect, most of our work does not rely on this fact, but it nevertheless indicates what matters.

\section{Axioms: Group 2}
\label{sec:Axioms group 2}Before we present the second group of axioms, let us define a notion familiar from classical first-order logic. We say that a formula $\varphi$ is a \emph{generalization} of a formula $\psi$ if for some $n<\omega$ there are variables $x_{1},\ldots,x_{n}$ such that $\varphi=\sup_{x_{1}}\!\!\cdots\sup_{x_{n}}\psi$ (see \cite{Enderton:MathematicalIntroductionToLogic}).  Observe that every formula is a generalization of itself.

When all generalizations of the following eight axiom schemata are added to all generalizations of (A1) -- (A6) (in the appropriate language of course), we obtain an axiomatization for continuous first-order logic.  In this article we show that this axiomatization is complete. Formal deductions and provability are defined as in $\mathsection\ref{sec: Axioms}$, the only rule of inference again being \emph{modus ponens}.  We write $\Gamma\vdash_{Q}\varphi$ to indicate that $\varphi$ is provable from $\Gamma$ in continuous first-order logic. Recall that a formula $\varphi$ is provable from $\Gamma$ just in case it is provable from a finite subset of $\Gamma$.  

 The first three schemata are analogues of axiom schemata of classical first-order logic (see \cite{Enderton:MathematicalIntroductionToLogic}).\\
\begin{list}{(A\arabic{Lcount})}
{\usecounter{Lcount}}
\setcounter{Lcount}{6}
  \item $(\sup_{x}\psi\dotminus\sup_{x}\varphi)\dotminus\sup_x (\psi\dotminus\varphi)$ 
  \item $\varphi[t/x]\dotminus\sup_x\varphi$, where substitution is correct
  \item $\sup_x\varphi\dotminus\varphi$, where $x$ is not free in $\varphi$\\
\end{list}
  If we are dealing with a continuous signature with a metric, we add:\\ 
\begin{list}{(A\arabic{Lcount})}
{\usecounter{Lcount}}
\setcounter{Lcount}{9}
  \item $dxx$
  \item $dxy\dotminus dyx$
  \item $(dxz\dotminus dxy)\dotminus dyz$
  \item[]
  \item For each $f\in\mathcal{F}$, $\epsilon\in(0,1]$, and $r,q\in\mathbb{D}$ with $r>\epsilon$ and $q<\delta_{f,i}(\epsilon)$, the following is an axiom (where $|\bar{x}|=i$ and $|\bar{x}| + |\bar{y}| + 1= n_{f}$):\\
 \,\\
   $(q\dotminus dzw)\wedge (df\bar{x}z\bar{y}f\bar{x}w\bar{y}\dotminus r)$\\
   
\item For each $P\in\mathcal{R}$, $\epsilon\in(0,1]$, and $r,q\in\mathbb{D}$ with $r>\epsilon$ and $q<\delta_{P,i}(\epsilon)$, the following is an axiom (where $|\bar{x}|=i$ and $|\bar{x}| + |\bar{y}| + 1= n_{P}$):\\
\,\\
 $(q\dotminus dzw)\wedge ((P\bar{x}z\bar{y}\dotminus P\bar{x}w\bar{y})\dotminus r)$\\
\end{list}

 Axiom schemata (A10) -- (A12) assert that $d$ is pseudo-metric.  These axiom schemata correspond to the equivalence relation axiom schemata of classical first-order logic.  Although less immediate, (A13) -- (A14) define the uniform continuity moduli with respect to the $i$th argument.  
 
 Let us informally consider (A13).  Suppose (A13) holds and that $d(z,w)<\delta_{f,i}(\epsilon)$.  Then there is $q\in\mathbb{D}$ such that $d(z,w)<q<\delta_{f,i}(\epsilon)$, so $q\dotminus d(z,w)>0$, whence by (A13) it follows that $d(f(\bar{x},z,\bar{y}),f(\bar{x},w,\bar{y}))\dotminus r=0$, i.e., $d(f(\bar{x},z,\bar{y}),f(\bar{x},w,\bar{y}))\leq r$ for every $r>\epsilon$.  Thus, $d(f(\bar{x},z,\bar{y}),f(\bar{x},w,\bar{y}))\leq\epsilon$, so $\delta_{f,i}$ is a uniform continuity modulus of $f$ with respect to the $i$th argument. Conversely, suppose $\delta_{f,i}$ is a uniform continuity modulus of $f$ with respect to the $i$th argument and that $r,q\in\mathbb{D}$ are such that $r>\epsilon$ and $q<\delta_{P,i}(\epsilon)$.  On the one hand, if $q\dotminus d(z,w)>0$, then $d(z,w)<q<\delta_{f,i}(\epsilon)$ and so $d(f(\bar{x},z,\bar{y}),f(\bar{x},w,\bar{y}))\leq \epsilon<r$, whence $d(f(\bar{x},z,\bar{y}),f(\bar{x},w,\bar{y}))\dotminus r=0$ and therefore $(q\dotminus d(z,w))\wedge (d(f(\bar{x},z,\bar{y}),f(\bar{x},w,\bar{y}))\dotminus r)=0$.  On the other hand, if $q\dotminus d(z,w)\leq 0$, again we have $(q\dotminus d(z,w))\wedge (d(f(\bar{x},z,\bar{y}),f(\bar{x},w,\bar{y}))\dotminus r)=0$.

\begin{rem} In view of Remark \ref{rem:congruence}, the reader may have noticed that a continuous signature could equivalently be defined so that $\mathcal{G}$ may instead have the form $\{\delta_{s}:(0,1]\to (0,1]: s\in\mathcal{R}\cup\mathcal{F}\}$, whereby one could forgo talk of moduli of uniform continuity of each symbol $s$ with respect to each of its arguments.  In particuar, (A13) could be replaced by 
\begin{itemize}
\item[(A13$^{\prime}$)]  For every $\epsilon\in(0,1]$ and $r,q\in\mathbb{D}$ with $r>\epsilon$ and $q<\Delta_{f}(\epsilon)$, the following is an axiom:\\
   $(q\dotminus\bigvee_{i<n_{f}} dx_{i}y_{i})\wedge (dfx_{0}\cdots x_{n_{f}-1}fy_{0}\cdots y_{n_{f}-1}\dotminus r)$
   \end{itemize}
   (A14) could also be replaced by an analogous axiom.   Situations might arise in which the uniform continuity moduli of a symbol with respect to each of its arguments are very different, and one might desire to keep track of this, thereby preferring one definition over the other.
\end{rem}
 
We conclude this section with a soundness theorem.
\begin{thm}[Soundness]\label{prop:sound} 
Let $\mathcal{L}$ be a continuous signature (with a metric), and let $\Gamma\subseteq \textup{For}(\mathcal{L})$.  Then for every formula $\varphi$,
\mbox{}
\begin{itemize}
\item[(i)] If $\Gamma\vdash_{Q}\varphi$, then $\Gamma\vDash_{Q}\varphi$ $(\Gamma\vDash_{QC}\varphi)$. 
\item[(ii)] If $\Gamma$ is $(\!\text{completely})$ satisfiable, then $\Gamma$ is consistent.
\end{itemize}	 
\end{thm}

\begin{proof} It suffices to show that each axiom schema is valid. The facts recorded in $\mathsection\ref{sec:bbt}$ can be used to show that (A1) -- (A6) are valid, while the results from \S \ref{sec:sub} can be used to prove that (A7) -- (A9) are valid. The discussion following (A14) explains why (A10) -- (A14) are valid.
\end{proof}

\section{The Essentials}
\label{sec:relations}
In this section, we rely heavily on facts recorded in $\mathsection\ref{sec:bbt}$ insofar as we exploit obvious relations amongst continuous first-order logic and continuous and  Łukasiewicz propositional logics. 
\subsection*{The Deduction Theorem, Generalization Theorem, and a Lemma about Bound Substitution}

Observe 
\begin{eqnarray*}
\{(P\dotminus Q)\dotminus Q, Q\}&\vdash^{\tiny\L}&\! P,\,\,\, \text{but}\\ 
\{(P\dotminus Q)\dotminus Q\}&\nvdash^{\tiny\L}&\! P\dotminus Q.
\end{eqnarray*}
 Thus the reader should be unsurprised to find that continuous first-order logic only satisfies a weak version of the classical Deduction Theorem.
\begin{thm}[Deduction Theorem]\label{thm:deduc}
Let $\Gamma\subseteq \textup{For}(\mathcal{L})$, and let $\varphi,\psi$ be formulae.\\
Then $\Gamma\cup\{\psi\}\vdash_{Q}\varphi$ if and only if $\Gamma\vdash_{Q}\varphi\dotminus n\psi$ for some $n<\omega$.
\end{thm}
\begin{proof}  Right to left is clear. For the implication from left to right, observe that for all $\alpha,\beta,\gamma$ and $n,m<\omega$, 
$$
\vdash^{\tiny\L}\!\ ((\beta\dotminus (n+m)\alpha)\dotminus((\beta\dotminus\gamma)\dotminus n\alpha))\dotminus(\gamma\dotminus m\alpha).
$$
This can be shown using the semantics of Łukasiewicz propositional logic. The proof then proceeds by induction on the set of consequences of $\Gamma\cup\{\psi\}$.
\end{proof}

\begin{lem}[Generalization Theorem]
\label{lem:Genthm} Let $\Gamma\subseteq \textup{For}(\mathcal{L})$, and let $\varphi$ be a formula. 
If $\Gamma\vdash\varphi$ and $x$ is not a free variable in $\Gamma$, then $\Gamma\vdash\sup_{x} \varphi$. 
\end{lem}

\begin{proof} This is similar to the proof of the Generalization Theorem for classical first-order logic (see \cite{Enderton:MathematicalIntroductionToLogic}).
\end{proof}

To illustrate how we rely on facts recorded in $\mathsection\ref{sec:bbt}$ , we provide a proof of the next lemma.
\begin{lem}
\label{lem:Bound substitution}
Let $\varphi$ be a formula, and let $y$ be a variable which does not occur in $\varphi$. Then    
\mbox{}
\begin{itemize}
\item[(a)] $\vdash_{Q}\varphi\dotminus\varphi \{y/x\}$.
\item[(b)] $\vdash_{Q}\varphi\{y/x\}\dotminus\varphi$.
\end{itemize}
\end{lem}

\begin{proof} This is proved by induction on $\varphi$ for $y$ which does not occur in $\varphi$.
\begin{itemize}
\item[] $\varphi=Pt_{0}\cdots t_{n_{P}-1}$: Then by definition $\varphi\{y/x\}=\varphi$; since $\vdash_{Q}\!\!\varphi\dotminus\varphi$, we are done. 
\item[] $\varphi=\frac12 \psi$: Then by definition $\varphi\{y/x\}=\frac12\psi\{y/x\}$ and this bound substitution is correct.  So by the induction hypothesis, we have $\vdash_{Q}\psi\dotminus\psi \{y/x\}$ and $\vdash_{Q}\psi\{y/x\}\dotminus\psi$. Now observe
$$
\vdash^{\tiny\text{C}\L}\!\!\frac12 \alpha\dotminus\alpha,\,\,\,\text{and}
$$
$$
\vdash^{\tiny\text{C}\L}\!\!(\frac12 \beta\dotminus\frac12\alpha)\dotminus\frac12(\beta\dotminus\alpha).
$$
Thus by a few applications of \emph{modus ponens}, we arrive at (a) and (b).
\item[] $\varphi=\neg\psi$: Then again by definition $\varphi\{y/x\}=\neg\psi\{y/x\}$ and this bound substitution is correct.  Observe
$$
\vdash^{\tiny\L}\!(\neg\alpha\dotminus\neg\beta)\dotminus(\beta\dotminus\alpha).
$$
Using the induction hypothesis and \emph{modus ponens} we obtain (a) and (b).
\item[] $\varphi=(\beta\dotminus\alpha)$: Then $\varphi\{y/x\}=(\beta\{y/x\}\dotminus\alpha\{y/x\})$ and $\beta\{y/x\}$ and $\alpha\{y/x\}$ are correct.  Our induction hypothesis supplies us 
$$
\vdash_{Q}\alpha\dotminus\alpha\{y/x\},
$$ 
$$
\vdash_{Q}\alpha\{y/x\}\dotminus\alpha,
$$ 
and
$$ 
\vdash_{Q}\beta\dotminus\beta\{y/x\},
$$ 
$$
\vdash_{Q}\beta\{y/x\}\dotminus\beta.
$$   
Moreover, observe
$$
\vdash^{\tiny\L}\!(((\beta'\dotminus\alpha')\dotminus(\beta\dotminus \alpha))\dotminus (\beta'\dotminus \beta))\dotminus (\alpha\dotminus\alpha').
$$
Hence, by repeated applications of \emph{modus ponens} and the induction hypothesis, we obtain (a) and (b).
\item[] $\varphi=\sup_{x}\psi$: Then $\varphi\{y/x\}=\sup_{x}\psi[y/x]$ and this substitution is correct ($y$ is not free in $\psi$).  Hence,
\begin{align*}
&\vdash_{Q}\psi[y/x]\dotminus\sup\!_{x}\psi\, (\text{A8})\\    
&\vdash_{Q}\sup\!_{y} (\psi[y/x]\dotminus\sup\!_{x}\psi)\, (\text{Generalization Theorem}) \\
&\vdash_{Q}(\sup\!_{y} \psi[y/x]\dotminus\sup\!_{y}\sup\!_{x}\psi)\dotminus \sup\!_{y} (\psi[y/x]\dotminus\sup\!_{x}\psi)\, (\text{A7})\\
&\vdash_{Q}\sup\!_{y} \psi[y/x]\dotminus\sup\!_{y}\sup\!_{x}\psi\,\, (\textit{modus ponens})\\
&\vdash_{Q}((\sup\!_{y} \psi[y/x]\dotminus\sup\!_{x}\psi)\dotminus(\sup\!_{y} \psi[y/x]\dotminus\sup\!_{y}\sup\!_{x}\psi))\dotminus (\sup\!_{y}\sup\!_{x}\psi\dotminus\sup\!_{x} \psi)\,(\text{A2})\\
&\vdash_{Q}\sup\!_{y}\sup\!_{x}\psi\dotminus\sup\!_{x}\psi\,(\text{A9})\, (y\, \text{is not free in}\, \sup\!_{x}\psi),
\end{align*}
from which it follows by applying \emph{modus ponens} several times that 
$\vdash_{Q}\varphi\{y/x\}\dotminus\varphi$.
A similar argument can be used to establish (a).
\item[] $\varphi=\sup_{z}\psi$ and $z\neq x$: Then $\varphi\{y/x\}=\sup_{z}\psi\{y/x\}$ and this bound substitution is correct.  Apply the induction hypothesis, the Generalization Theorem, (A7), and \emph{modus ponens} to obtain (a) and (b).  \qedhere
\end{itemize}
\end{proof}

If $c$ is an $0$-ary function symbol (what we of course call a \emph{constant symbol}) and $x$ is a variable, we may define $\varphi[x/c]$ in the expected way.  However, we omit a formal definition and state a lemma without proof.
\begin{lem}\label{lem:csub} Let $\Gamma\subseteq \textup{For}(\mathcal{L})$, and let $\varphi$ be a formula.
If $\Gamma\vdash_{Q}\varphi$, $c$ is a constant symbol that does not occur in $\Gamma$, and $x$ does not occur in $\varphi$, then $\Gamma\vdash_{Q}\sup_{x}\varphi[x/c]$.
\end{lem}

\subsection*{Maximal Consistent Sets of Formulae}
 In classical first-order logic, we say that a set $\Delta$ of formulae is \emph{maximal consistent} if $\Delta$ is consistent and for every formula $\varphi$,
\begin{center}
$\varphi\in \Delta$ or $\neg\varphi\in\Delta$.
\end{center} 
In continuous first-order, however, this definition will not do.  For we certainly  should like to exclude \emph{both} $\frac{1}{2}$ \emph{and} $\neg\frac{1}{2}$ from \emph{every} consistent set. In this subsection, we show what one \emph{should} mean by calling a set of formulae maximal consistent in continuous first-order logic.

\begin{lem} \label{lem:max} Let $\Gamma\subseteq \textup{For}(\mathcal{L})$ be consistent.
Then for all formulae $\varphi,\psi$,
\mbox{}
\begin{itemize}
\item[(i)] If $\Gamma\vdash_{Q}\varphi\dotminus 2^{-n}$ for all $n<\omega$ then $\Gamma\cup\{\varphi\}$ is consistent.
\item[(ii)] Either $\Gamma\cup\{\varphi\dotminus\psi\}$ or $\Gamma\cup\{\psi\dotminus\varphi\}$ is consistent.
\end{itemize}
\end{lem}

\begin{proof}
\mbox{}
\begin{itemize}
\item[(i)] Suppose $\Gamma\vdash_{Q}\varphi\dotminus 2^{-n}$ for all $n<\omega$.  For \emph{reductio ad absurdum}, assume $\Gamma\cup\{\varphi\}$ is inconsistent.  Then by the Deduction Theorem, $\Gamma\vdash1\dotminus m\varphi$ for some $m<\omega$.  Observe 
\begin{eqnarray*}
&\vdash^{\tiny\L}&\!(((\beta\dotminus\alpha_{\sigma(0)})\dotminus\alpha_{\sigma(1)})\dotminus\cdots\dotminus\alpha_{\sigma(k-1)})\dotminus (((\beta\dotminus\alpha_{0})\dotminus\alpha_{1})\dotminus\cdots\dotminus\alpha_{k-1})
\end{eqnarray*}
 for any $k<\omega$ and permutation $\sigma$ on $k$. By repeated uses of this fact, axiom (A2), and our supposition, we find that $\Gamma\vdash_{Q}((1\dotminus 2^{-1})\dotminus 2^{-2})\dotminus\cdots\dotminus 2^{-m}$, so $\Gamma$ is inconsistent, yielding a contradiction.
\item[(ii)] Suppose $\Gamma\cup\{\varphi\dotminus\psi\}$ is inconsistent.  Then by the Deduction Theorem, there is $n<\omega$ such that $\Gamma\vdash_{Q} 1\dotminus n(\varphi\dotminus\psi)$.  Observe that for all $m<\omega$, $\vdash^{\tiny\L}\! (\beta\dotminus\alpha)\dotminus(1\dotminus m(\alpha\dotminus\beta))$, whereby it follows that $\vdash_{Q} (\varphi\dotminus\psi)\dotminus(1\dotminus n(\varphi\dotminus\psi))$.  Hence, by \emph{modus ponens} we have $\Gamma\vdash_{Q}\psi\dotminus\varphi$, so $\Gamma\cup\{\psi\dotminus\varphi\}$ is consistent. 
\end{itemize}
\end{proof}

\begin{lem}
\label{lem:Partially}
Let $\Gamma\subseteq \textup{For}(\mathcal{L})$ be consistent, and let $E(\Gamma)\equals\{\Gamma^{\prime} :\Gamma\subseteq\Gamma^{\prime}\subseteq \textup{For}(\mathcal{L})\,\textup{and}\,\,\Gamma^{\prime}\,\textup{is consistent}\}$.  
 Then $\Delta$ is a maximal member of $E(\Gamma)$ if and only if for all formulae $\varphi,\psi$, 
 \mbox{}
 \begin{itemize}
 \item[(i)] $\Gamma\subseteq\Delta$ and $\Delta$ is consistent.
\item[(ii)] If $\Delta\vdash_{Q}\varphi\dotminus 2^{-n}$ for all $n<\omega$ then $\varphi\in\Delta$.
\item[(iii)] $\varphi\dotminus\psi\in\Delta$ or $\psi\dotminus\varphi\in\Delta$.
\end{itemize}
\end{lem}

\begin{proof} \mbox{}
\begin{itemize} 
\item[($\Rightarrow$)]  If $\Delta$ is a maximal member of $E(\Gamma)$, then (i) is immediate, while (ii) and (iii) follow from Lemma \ref{lem:max}.
\item[($\Leftarrow$)] Suppose (i), (ii),  and (iii) hold.  Let $\Gamma^{\prime}\in E(\Gamma)$ be such that $\Delta\subseteq\Gamma^{\prime}$. Assume $\varphi\in\Gamma^{\prime}$. Then for every $n<\omega$, $2^{-n}\dotminus\varphi\notin\Gamma^{\prime}$ and so $\varphi\dotminus 2^{-n}\in\Delta$. It follows that $\varphi\in\Delta$, as desired.\qedhere
\end{itemize}
\end{proof}

\begin{rem}
In continuous propositional logic, conditions (ii) and (iii) are independent.  To see this, let $\mathcal{S}$ be a continuous propositional logic. First we show that (ii) does not imply (iii).  Let $\Delta\equals\{\varphi\in\mathcal{S}:\,\,\vdash^{\!\tiny{\text{C} \L}}\varphi\}$.  Observe that by completeness of continuous propositional logic (Fact \ref{fact:finite}), $\Delta$ is consistent, and if $\Delta\vdash^{\tiny\text{C}\L}\!\!\varphi\dotminus 2^{-n}$ for all $n<\omega$, then $\varphi\in \Delta$.  However, since $\nvdash^{\!\tiny{\text{C} \L}} P\dotminus\frac{1}{2}$ and $\nvdash^{\!\tiny{\text{C} \L}}\frac{1}{2}\dotminus P$, neither $P\dotminus\frac{1}{2}$ nor $\frac{1}{2}\dotminus P$ is in $\Delta$.  We now show that (iii) does not imply (ii).  Let $\Gamma\equals\{P\dotminus 2^{-n}:n<\omega\}$. Then $\Gamma$ is consistent, and so by (ii) of Lemma \ref{lem:max} we can construct a consistent set $\Delta\supseteq\Gamma$ such that for every $\varphi,\psi\in \mathcal{S}$, either $\varphi\dotminus\psi\in\Delta$ or $\psi\dotminus\varphi\in\Delta$. Nonetheless, $\Delta\vdash^{\!\tiny{\text{C} \L}}P\dotminus 2^{-n}$ for all $n<\omega$, yet $P\notin\Delta$.
\end{rem}
Lemma \ref{lem:Partially} justifies the following the definition.

\begin{defn}
\label{defn:max} Let $\Delta\subseteq \textup{For}(\mathcal{L})$.  We say that $\Delta$ is \emph{maximal consistent} if $\Delta$ is consistent and for all formulae $\varphi,\psi$,
\mbox{}
\begin{itemize}
\item[(i)] If $\Delta\vdash_{Q}\varphi\dotminus 2^{-n}$ for all $n<\omega$ then $\varphi\in\Delta$.
\item[(ii)] $\varphi\dotminus\psi\in\Delta$ or $\psi\dotminus\varphi\in\Delta$.
\end{itemize}
\end{defn}
Condition (i) says that if $\Delta$ can furnish a series of proofs which altogether indicate that $\varphi$ is a consequence of $\Delta$, then $\Delta$ must be granted $\varphi$.  Condition (ii) says that $\Delta$ can compare any two formulae.  

\begin{rem}Observe that if $\Delta$ a maximal consistent set of formulae, then $\varphi\in\Delta$ if and only if $\Delta\vdash_{Q}\varphi$.  It can be shown that in the above definition condition (i) can be replaced by condition (i$^{\prime}$): If $\{\varphi\dotminus 2^{-n}: n<\omega\}\subseteq\Delta$, then $\varphi\in\Delta$.  Indeed, condition (i$^{\prime}$) is also independent of (ii) in continuous propositional logic. 
\end{rem}

We thereby obtain the following result.
\begin{thm}
\label{thm:zorn}
If $\Gamma\subseteq \textup{For}(\mathcal{L})$ is consistent, then there exists a maximal consistent set of formulae $\Delta\supseteq\Gamma$.
  \end{thm}

\begin{proof} By Zorn's Lemma and Lemma \ref{lem:Partially}.
\end{proof}

\begin{lem}\label{lem:wit}
Let $\Gamma\subseteq \textup{For}(\mathcal{L})$, let $\Delta$ be a maximal consistent set of formulae containing $\Gamma$, let $t$ be a term, and let $p\in\mathbb{D}$.  Then if $\sup_{x}\varphi\dotminus p\in\Delta$, there is  a formula $\varphi^{\prime}$ such that $\varphi\equiv\varphi^{\prime}$ and $\varphi^{\prime}[t/x]\dotminus p\in\Delta$.
\end{lem}

\begin{proof}
By Lemma \ref{lem:bdsub},  there is a formula $\varphi'$ such that no variable of $t$ is bound in $\varphi'$ and $\varphi\equiv\varphi'$, so by repeated applications of Lemma \ref{lem:Bound substitution} and (A2), we have $\vdash_{Q}\varphi'\dotminus\varphi$.  Then by the Generalization Theorem, $\vdash_{Q}\sup_{x}(\varphi'\dotminus\varphi)$, whence by (A7) it follows that $\vdash_{Q}\sup_{x}\varphi'\dotminus\sup_{x}\varphi$.  Now as $\sup_{x}\varphi\dotminus p\in\Delta$, by (A2) we have $\Delta\vdash_{Q}\sup_{x}\varphi'\dotminus p$.  Because $\varphi'[t/x]$ is correct,  again by (A2) and by (A8) we have $\Delta\vdash_{Q}\varphi'[t/x]\dotminus\sup_{x} \varphi$,  so once more by (A2) it follows that $\Delta\vdash_{Q}\varphi'[t/x]\dotminus p$. Therefore, since $\Delta$ is maximal consistent, we conclude that $\varphi'[t/x]\dotminus p\in\Delta$, as desired.    
\end{proof}

We will use the following lemma to define a continuous pre-structure and to help us in parts of our proof of completeness.  
\begin{lem}\label{lem:def}
Let $\Gamma\subseteq \textup{For}(\mathcal{L})$ be consistent, let $\Delta$ be a maximal consistent set of formulae containing $\Gamma$, and let $\varphi$ be a formula. Then\\
$$
\sup\{p\in\mathbb{D}:p\dotminus \varphi\in\Delta\}=\inf\{q\in\mathbb{D}:\varphi\dotminus q\in\Delta\}.
$$
\end{lem}
\begin{proof} Straightforward, by Lemma \ref{lem:Partially}.
\end{proof}

\section{Completeness Theorem}

In classical first-order logic, we say that a set $\Gamma$ of formulae is \emph{Henkin complete} if for every formula $\varphi$ and variable $x$, there is a constant $c$ such that
\begin{center}
$\neg\forall x\varphi\to\neg\varphi[c/x]\in\Gamma$.
\end{center}
The intuition to define a Henkin complete set in continuous first-order logic as the syntactic translation of this definition will not work.  Indeed, if $\Gamma$ is maximal consistent and $\neg\sup_{x}\neg\varphi\in\Gamma$, i.e., $\inf_{x}\varphi\in\Gamma$, we would be better off to not require that there is a constant $c$ such that $\varphi[c/x]\in\Gamma$. For example, a perfectly consistent situation could arise in which $\inf_{x}\varphi\in\Gamma$, there is a sequence of formulae $(\varphi[t_{i}/x]:i<\omega)$ such that $\Gamma\vdash_{Q}\varphi[t_{n}/x]\dotminus 2^{-n}$ for each $n<\omega$, yet each term $t$ has a non-zero  $p_{t}\in\mathbb{D}$ such that $\Gamma\vdash_{Q}p_{t}\dotminus\varphi[t/x]$, whereby we should certainly not require that there is a term $t$ such that $\varphi[t/x]\in\Gamma$. Rather than attempt to prevent this sort of situation from ever occurring, we wish to accommodate this sort of situation in our definition of a Henkin complete set.

\begin{defn}\label{defn:henkin} Let $\Gamma$ be a set of formulae.  We say that $\Gamma$ is \emph{Henkin complete} if
for every formula $\varphi$, variable $x$, and $p,q\in\mathbb{D}$ with $p<q$, there is a constant symbol $c$ such that
\begin{center}
 $(\sup_{x}\varphi\dotminus q)\wedge(p\dotminus\varphi[c/x])\in\Gamma$.
\end{center}
\end{defn}

The next proposition guarantees we can construct a consistent Henkin complete set.
\begin{prop}
 Let $\mathcal{L}$ be a continuous signature.  There exists a continuous signature $\mathcal{L}^{c}$ with $\mathcal{L}\subseteq\mathcal{L}^{c}$ and a set $\Gamma^{c}\subseteq For(\mathcal{L}^{c})$ such that for every $\Gamma\subseteq \textup{For}(\mathcal{L})$,
\mbox{}
\begin{itemize}
\item[(i)] $\Gamma\cup\Gamma^{c}$ is Henkin complete.
\item[(ii)] If $\Gamma$ is consistent, then $\Gamma\cup\Gamma^{c}$ is consistent.  
\end{itemize}
\end{prop}

\begin{proof}  Permitting some abuse of notation, we define an increasing sequence of signatures inductively:
\mbox{}
\begin{itemize}
\item[] $\mathcal{L}_{0}\equals\mathcal{L}$
\item[] $\mathcal{L}_{n+1}\equals\mathcal{L}_{n}\cup\{c_{(\varphi,x,p,q)}:\varphi\in For(\mathcal{L}_{n}),x\in V,\text{ and } p,q\in\mathbb{D} \text{ with } p<q\}$ 
\end{itemize}
where each $c_{(\varphi,x,p,q)}$ is a new constant symbol.
We then put
\begin{itemize}
\item[] $\mathcal{L}^{c}\equals\bigcup_{n<\omega}\!\mathcal{L}_{n}$
\item[] $\psi_{(\varphi,x,p,q)}\equals(\sup_{x}\varphi\dotminus q)\wedge(p\dotminus\varphi[c_{(\varphi,x,p,q)}/x])$ and
\item[] $\Gamma^{c}\equals\{\psi_{(\varphi,x,p,q)}: \varphi\in For(\mathcal{L}^{c}),x\in V,\text{ and } p,q\in\mathbb{D} \text{ with } p<q\}$.
\end{itemize}
(i) is obviously true, so we proceed to establish (ii).

Suppose $\Gamma$ is consistent.  For \emph{reductio ad absurdum}, assume $\Gamma\cup\Gamma^{c}$ is inconsistent.  Then there is a minimal finite non-empty $\Gamma^{c}_{0}\subseteq\Gamma^{c}$ such that $\Gamma\cup\Gamma^{c}_{0}$ is inconsistent.  We may write $\Gamma^{c}_{0}=\{\psi_{(\varphi_{i},x_{i},p_{i},q_{i})}:i<n\}$, where the $(\varphi_{i},x_{i},p_{i},q_{i})$ are pairwise distinct.  For each $i<n$, there is a minimal $m_{i}$ such that $\varphi_{i}\in For(\mathcal{L}_{m_{i}})$. Let $m_{i_{0}}$ be the maximal such $m_{i}$.  For simplicity of notation, put $\varphi\equals\varphi_{i_0}$, $x\equals x_{i_0}$, $p\equals p_{i_0}$, $q\equals q_{i_0}$, and $m\equals m_{i_0}$.

Now observe that for each $i<n$, $\varphi_{i}\in For(\mathcal{L}_{m})$, while $c_{(\varphi,x,p,q)}\in\mathcal{L}_{m+1}\backslash\mathcal{L}_{m}$.  Put $\Gamma^{c}_{1}\equals\Gamma^{c}_{0}\backslash\{\psi_{(\varphi,x,p,q)}\}$ and $\Gamma_{1}\equals\Gamma\cup\Gamma^{c}_{1}$.  Note that $c_{(\varphi,x,p,q)}$  does not occur in $\Gamma_{1}$ and that $\Gamma_{1}$ must be consistent (by the minimality of $\Gamma^{c}_{0}$).  Nonetheless, $\Gamma\cup\Gamma^{c}_{0}=\Gamma_{1}\cup\{(\sup_{x}\varphi\dotminus q)\wedge(p\dotminus\varphi[c_{(\varphi,x,p,q)}/x])\}$ is inconsistent. Observe that $\vdash^{\tiny{\!\L}}(\alpha\wedge\beta)\dotminus\alpha$ and $\vdash^{\tiny{\!\L}}(\alpha\wedge\beta)\dotminus\beta$.  Thus, by \emph{modus ponens}, 
\begin{eqnarray*}
\Gamma_{1}\cup\{(\sup\!_{x}\varphi\dotminus q)\}&\vdash_{Q}&(\sup\!_{x}\varphi\dotminus q)\wedge(p\dotminus\varphi[c_{(\varphi,x,p,q)}/x]),\,\,\,\text{and}\\ 
\Gamma_{1}\cup\{(p\dotminus\varphi[c_{(\varphi,x,p,q)}/x])\}&\vdash_{Q}&(\sup\!_{x}\varphi\dotminus q)\wedge(p\dotminus\varphi[c_{(\varphi,x,p,q)}/x]).
\end{eqnarray*}
 Since $\Gamma\cup\Gamma^{c}_{0}$ is inconsistent, by the Deduction Theorem there is $n<\omega$ such that $\Gamma_{1}\vdash_{Q}1\dotminus n((\sup_{x}\varphi\dotminus q)\wedge(p\dotminus\varphi[c_{(\varphi,x,p,q)}/x]))$. Therefore, by \emph{modus ponens}, we find that  $\Gamma_{1}\cup\{(\sup_{x}\varphi\dotminus q)\}$ and $\Gamma_{1}\cup\{(p\dotminus\varphi[c_{(\varphi,x,p,q)}/x])\}$ are both inconsistent. It follows that $\Gamma_{1}\vdash_{Q}(q\dotminus \sup_{x}\varphi)$ and $\Gamma_{1}\vdash_{Q}(\varphi[c_{(\varphi,x,p,q)}/x]\dotminus p)$ (cf. proof of Lemma \ref{lem:max}).  

Now let $y$ be a variable not occurring in $\varphi$.  On the one hand, by Lemma \ref{lem:Bound substitution} we have $\Gamma_{1}\vdash_{Q}(q\dotminus \sup_{x}\varphi)\{y/x\}\dotminus (q\dotminus\sup_{x}\varphi)$, so by \emph{modus ponens}, 
\begin{eqnarray*}
\Gamma_{1}\vdash_{Q}(q\dotminus\sup\!_{x}\varphi)\{y/x\}&=&(q\{y/x\}\dotminus \sup\!_{y}\varphi[y/x])\\
&=&(q\dotminus\sup\!_{y}\varphi[y/x]).
\end{eqnarray*}
 On the other hand, because $c_{(\varphi,x,p,q)}$  does not occur in $\Gamma_{1}$, by Lemma \ref{lem:csub} we have 
 \begin{eqnarray*}
 \Gamma_{1}\vdash_{Q}\sup\!_{y}(\varphi[c_{(\varphi,x,p,q)}/x]\dotminus p)[y/c_{(\varphi,x,p,q)}]&=&\sup\!_{y}(\varphi[c_{(\varphi,x,p,q)}/x][y/c_{(\varphi,x,p,q)}]\dotminus p[y/c_{(\varphi,x,p,q)}])\\
 &=&\sup\!_{y}(\varphi[y/x]\dotminus p). 
 \end{eqnarray*}
 Then by (A7) and \emph{modus ponens}, $\Gamma_{1}\vdash_{Q}\sup_{y}\varphi[y/x]\dotminus \sup_{y} p$.  Since $y$ is not free in $p$, (A9) yields $\Gamma_{1}\vdash_{Q}\sup_{y} p\dotminus p$.  Thus by (A2) and \emph{modus ponens}, $\Gamma_{1}\vdash_{Q}\sup_{y}\varphi[y/x]\dotminus p$.

It remains to observe that since $\Gamma_{1}\vdash_{Q}(q\dotminus\sup_{y}\varphi[y/x])$ and $\Gamma_{1}\vdash_{Q}\sup_{y}\varphi[y/x]\dotminus p$,  by (A2) we have $\Gamma_{1}\vdash_{Q}q\dotminus p$, so $\Gamma_{1}$ is inconsistent, which is impossible ($q\dotminus p\in\mathbb{D}$ and $q\dotminus p>0$ by construction).   
\end{proof}

\begin{defn}  We define the \emph{rank} of $\varphi$, $\emph{rank}(\varphi)$, by recursion:
\mbox{}
\begin{itemize}
\item[] $\emph{rank}(Pt_{0}\cdots t_{n_{P}-1})\equals0$.
\item[] $\emph{rank}(\frac{1}{2}\varphi)\equals \emph{rank}(\varphi)+1$.
\item[] $\emph{rank}(\neg\varphi)\equals \emph{rank}(\varphi)+1$.
\item[] $\emph{rank}(\varphi\dotminus\psi)\equals\emph{rank}(\varphi)+\emph{rank}(\psi)+1$.
\item[] $\emph{rank}(\sup_{x}\varphi)\equals\emph{rank}(\varphi)+1$.
\end{itemize}
\end{defn}

\begin{thm}\label{thm:precomp}
Let $\mathcal{L}$ be a continuous signature (possibly with a metric), and let $\Gamma\subseteq \textup{For}(\mathcal{L})$. Assume $\Gamma$ is consistent.  Then there is a continuous signature $\mathcal{L}^{c}\supseteq\mathcal{L}$, a maximal consistent  and Henkin complete set $\Delta\subseteq For(\mathcal{L}^{c})$ with $\Delta\supseteq\Gamma$, a continuous $\mathcal{L}^{c}$-pre-structure $\mathfrak{M}$, and an $\mathfrak{M}$-assignment $\sigma$ such that $(\mathfrak{M},\sigma)\vDash_{Q}\Delta$.
\end{thm}

\begin{proof}
 Suppose $\Gamma$ is consistent.  By the previous proposition, there is $\Gamma^{c}\subseteq\ For(\mathcal{L}^{c})$ such that $\Gamma\cup\Gamma^{c}$ is consistent and Henkin complete, so by Theorem \ref{thm:zorn} there is a maximal consistent $\Delta\subseteq For(\mathcal{L}^{c})$ such that $\Delta\supseteq\Gamma\cup\Gamma^{c}$.  Define a continuous $\mathcal{L}^{c}$-pre-structure $\mathfrak{M}$ such that $M$ is the set of all $\mathcal{L}^{c}$-terms. 
\mbox{}
\begin{itemize} 
\item[(i)] For each $f\in\mathcal{F}$, define $f^{\mathfrak{M}}:M^{n_{f}}\to M$ by setting \\
$f^{\mathfrak{M}}(t_{0},\ldots,t_{n_{f}-1})\equals ft_{0}\cdots t_{n_{f}-1}$ for all $t_{0},\ldots,t_{n_{f}-1}\in M.$
\item[(ii)] For each $P\in\mathcal{R}$, define  $P^{\mathfrak{M}}:M^{n_{P}}\to[0,1]$ by setting \\
$P^{\mathfrak{M}}(t_{0},\ldots,t_{n_{P}-1})\equals\sup\{p\in\mathbb{D}:p\dotminus Pt_{0}\cdots t_{n_{P}-1}\in\Delta\}$\\ for all $t_{0},\ldots,t_{n_{P}-1}\in M$.
\end{itemize}
Define an $\mathfrak{M}$-assignment $\sigma$ by setting $\sigma(x)\equals x$ for all $x\in V$.  It is a simple matter to check that $t^{\mathfrak{M},\sigma}=t$ for all terms $t$. 

It suffices to prove by induction on the number of quantifiers and connectives in $\varphi$, i.e., the rank of $\varphi$, 
that $\mathfrak{M}(\varphi,\sigma)=\sup\{p\in\mathbb{D}:p\dotminus\varphi \in\Delta\}$ for all $\varphi\in For(\mathcal{L}^{c})$.  For the sake of brevity, here we only consider the subcases $\varphi=\frac{1}{2}\psi$ and $\varphi=\sup_{x}\psi$.  Although the case for $\emph{rank}(\varphi)=0$ is trivial, we encourage the reader to work out the subcases of $\neg$ and $\dotminus$ for $\emph{rank}(\varphi)>0$.\\

\noindent $\varphi=\frac{1}{2}\psi$:  We wish to show that $\mathfrak{M}(\frac{1}{2}\psi,\sigma)=\sup\{p\in\mathbb{D}:p\dotminus\frac{1}{2}\psi \in\Delta\}$. By definition, $\mathfrak{M}(\frac{1}{2}\psi,\sigma)=\frac{1}{2}\sup\{p\in\mathbb{D}:p\dotminus\psi \in\Delta\}$, and by induction hypothesis, $\mathfrak{M}(\psi,\sigma)=\sup\{p\in\mathbb{D}:p\dotminus\psi \in\Delta\}$.  We must therefore show $\frac{1}{2}\sup\{p\in\mathbb{D}:p\dotminus\psi \in\Delta\}=\sup\{p\in\mathbb{D}:p\dotminus\frac{1}{2}\psi \in\Delta\}$.  We consider two cases.
\mbox{}
\begin{itemize}
\item[(a)] By Fact \ref{fact:finite}, since $\Delta$ is maximal consistent, for every $p\in\mathbb{D}$, $p\dotminus\frac{1}{2}\psi\in\Delta$ only if $\neg (\neg p\dotminus p)\dotminus\psi\in \Delta$.  Hence, for every $p\in\mathbb{D}$ such that $p\dotminus\frac{1}{2}\psi\in\Delta$, $p\leq\frac{1}{2}\neg(\neg p\dotminus p)\leq\frac{1}{2}\sup\{p\in\mathbb{D}:p\dotminus\psi \in\Delta\}$.  It follows that $\sup\{p\in\mathbb{D}:p\dotminus\frac{1}{2}\psi \in\Delta\}\leq\frac{1}{2}\sup\{p\in\mathbb{D}:p\dotminus\psi \in\Delta\}$.

\item[(b)] By Fact \ref{fact:finite}, since $\Delta$ is maximal consistent, for every $p\in\mathbb{D}$, $p\dotminus\psi\in\Delta$ only if $\frac{1}{2}p\dotminus\frac{1}{2}\psi\in\Delta$ (cf. proof of Lemma \ref{lem:Bound substitution}).   Hence, for every $p\in\mathbb{D}$ such that $p\dotminus\psi\in\Delta$, $\frac{1}{2}p\leq\sup\{p\in\mathbb{D}:p\dotminus\frac{1}{2}\psi\in\Delta\}$ and therefore $p\leq2\sup\{p\in\mathbb{D}:p\dotminus\frac{1}{2}\psi\in\Delta\}$.  It follows that $\sup\{p\in\mathbb{D}:p\dotminus\psi \in\Delta\}\leq2\sup\{p\in\mathbb{D}:p\dotminus\frac{1}{2}\psi \in\Delta\}$, whence $\frac{1}{2}\sup\{p\in\mathbb{D}:p\dotminus\psi \in\Delta\}\leq\sup\{p\in\mathbb{D}:p\dotminus\frac{1}{2}\psi \in\Delta\}$, as desired.
\end{itemize}

 \noindent $\varphi=\sup_{x}\psi$: We wish to show that $\mathfrak{M}(\sup_{x}\psi,\sigma)=\sup\{p\in\mathbb{D}:p\dotminus\sup_{x}\psi \in\Delta\}$.  Observe that by definition $\mathfrak{M}(\sup_{x}\psi,\sigma)=\sup\{\mathfrak{M}(\psi,\sigma^{t}_{x}): t\in M\}$.  So it suffices to show that $\sup\{\mathfrak{M}(\psi,\sigma^{t}_{x}): t\in M\}=\sup\{p\in\mathbb{D}:p\dotminus\sup_{x}\psi \in\Delta\}$. For \emph{reductio ad absurdum}, assume this equality fails to hold.  We of course consider two cases.
\mbox{}
\begin{itemize}
\item[(a)] Suppose $\sup\{\mathfrak{M}(\psi,\sigma^{t}_{x}): t\in M\}<\sup\{p\in\mathbb{D}:p\dotminus\sup_{x}\psi \in\Delta\}$.  Then for some $p,q\in\mathbb{D}$,  $\sup\{\mathfrak{M}(\psi,\sigma^{t}_{x}): t\in M\}<p<q<\sup\{p\in\mathbb{D}:p\dotminus\sup_{x}\psi \in\Delta\}$.    Since $\Gamma^{c}\subseteq\Delta$ and $p<q$, there is a constant $c$ such that $(\sup_{x}\psi\dotminus q)\wedge(p\dotminus\psi[c/x])\in\Delta$. Furthermore, there is $r\in\mathbb{D}$ such that $q<r$ and $r\dotminus\sup_{x}\psi\in\Delta$. As $\Delta$ is maximal consistent, by Fact \ref{fact:finite} it follows that $p\dotminus\psi[c/x]\in\Delta$.  Since $\emph{rank}(\psi[c/x])<\emph{rank}(\varphi)$, our induction hypothesis tells us that $\mathfrak{M}(\psi[c/x],\sigma)=\sup\{p\in\mathbb{D}:p\dotminus\psi[c/x] \in\Delta\}$.  Also, since $\psi[c/x]$ is correct, by Lemma \ref{lem:inter} it follows that $\mathfrak{M}(\psi[c/x],\sigma)=\mathfrak{M}(\psi,\sigma^{c}_{x})$.  Thus $p\leq\mathfrak{M}(\psi,\sigma^{c}_{x})\leq\sup\{\mathfrak{M}(\psi,\sigma^{t}_{x}): t\in M\}$, yielding a contradiction.
\item[(b)] Suppose $\sup\{p\in\mathbb{D}:p\dotminus\sup_{x}\psi \in\Delta\}<\sup\{\mathfrak{M}(\psi,\sigma^{t}_{x}): t\in M\}$.  On the one hand, since $\Delta$ is maximal consistent, for all $p\in\mathbb{D}$ such that $p>\sup\{p\in\mathbb{D}:p\dotminus\sup_{x}\psi \in\Delta\}$, $\sup_{x} \psi\dotminus p\in\Delta$.  On the other hand, there must be $t\in M$ such that $\mathfrak{M}(\psi,\sigma^{t}_{x})>\sup\{p\in\mathbb{D}:p\dotminus\sup_{x}\psi \in\Delta\}$, so there is $p\in\mathbb{D}$ such that $\mathfrak{M}(\psi,\sigma^{t}_{x})>p>\sup\{p\in\mathbb{D}:p\dotminus\sup_{x}\psi \in\Delta\}$.  Thus $\sup_{x} \psi\dotminus p\in\Delta$, whence by Lemma \ref{lem:wit}, there is $\psi'\equiv\psi$ such that $\psi'[t/x]\dotminus p\in\Delta$.  Thus $\inf\{p:\psi'[t/x]\dotminus p\in\Delta\}\leq p$. By Lemma \ref{lem:def}, $\inf\{p:\psi'[t/x]\dotminus p\in\Delta\}=\sup\{p\in\mathbb{D}:p\dotminus\psi'[t/x]\in\Delta\}$, and because $rank(\psi'[t/x])<\emph{rank}(\varphi)$, our induction hypothesis tells us that $\mathfrak{M}(\psi'[t/x],\sigma)=\sup\{p\in\mathbb{D}:p\dotminus\psi'[t/x]\in\Delta\}$.  Moreover, since $\psi'[t/x]$ is correct, $\mathfrak{M}(\psi'[t/x],\sigma)=\mathfrak{M}(\psi',\sigma^{t}_{x})$ by Lemma \ref{lem:inter}.  Finally, since $\psi'\equiv\psi$, we have $\mathfrak{M}(\psi',\sigma^{t}_{x})=\mathfrak{M}(\psi,\sigma^{t}_{x})$.  It follows that $\mathfrak{M}(\psi,\sigma^{t}_{x})\leq p,$ yielding another contradiction.     
\end{itemize} 
Now if $\varphi\in\Delta$, then by (A1), $\varphi\dotminus 0\in\Delta$, so $\mathfrak{M}(\varphi,\sigma)=\sup\{p\in\mathbb{D}:p\dotminus\varphi \in\Delta\}=\inf\{q\in\mathbb{D}:\varphi\dotminus q \in\Delta\}=0$ (Lemma \ref{lem:def}).  This shows that $(\mathfrak{M},\sigma)\vDash_{Q}\Delta$, proving our theorem. 
\end{proof}

We may now state and prove the long-awaited completeness theorem.
\begin{thm}[Completeness for Continuous First-Order Logic]\label{thm:comp}\,\ \\
Let $\mathcal{L}$ be a continuous signature, and let $\Gamma\subseteq \textup{For}(\mathcal{L})$.  Then $\Gamma$ is consistent only if $\Gamma$ is satisfiable.  Furthermore, if $\mathcal{L}$ is a continuous signature with a metric, then $\Gamma$ is consistent only if $\Gamma$ is completely satisfiable.
\end{thm}

\begin{proof}  By Theorem \ref{thm:precomp}, there is a continuous signature $\mathcal{L}^{c}\supseteq\mathcal{L}$, a maximal consistent set $\Delta\subseteq For(\mathcal{L}^{c})$ with $\Delta\supseteq\Gamma$, an $\mathcal{L}^{c}$-pre-structure $\mathfrak{M}$, and an $\mathfrak{M}$-assignment $\sigma$ such that $(\mathfrak{M},\sigma)\vDash_{Q}\Delta$.   If $\mathcal{L}$ does not have a metric, we are done; we observe  $(\mathfrak{M},\sigma)$ models $\Gamma$ and simply restrict $\mathfrak{M}$ to our original signature.  But if $\mathcal{L}$ has a metric $d$, we can only guarantee that $d^{\mathfrak{M}}$ is a pseudo-metric relative to this restriction.  We wish to find a continuous $\mathcal{L}$-structure $\widehat{\mathfrak{M}}$ and an $\widehat{\mathfrak{M}}$-assignment $\widehat{\sigma}$ such that $(\widehat{\mathfrak{M}},\widehat{\sigma})\vDash_{Q}\Gamma$.  By Theorem \ref{thm:completion}, there is a continuous $\mathcal{L}$-structure $\mathfrak{\widehat{M}}$ and an $\mathcal{L}$-morphism $\widehat{h}:M\to\widehat{M}$ such that for every formula $\varphi$, $\mathfrak{M}(\varphi,\sigma)=\mathfrak{\widehat{M}}(\varphi,h\circ\sigma).$  Therefore, putting $\widehat{\sigma}\equals h\circ\sigma$, since $(\mathfrak{M},\sigma)$ models $\Gamma$, $(\mathfrak{\widehat{M}},\widehat{\sigma})$ models $\Gamma$.
\end{proof}

We conclude with several corollaries.  The following result has a counterpart in \cite{Hay:1963} (see also \cite{Hajek:FuzzyLogic}).

\begin{cor}[Approximated Strong Completeness for Continuous Logic] 
Let $\mathcal{L}$ be a continuous signature, let$\Gamma\subseteq \textup{For}(\mathcal{L})$, and let $\varphi$ be a formula.  Then $\Gamma\vDash_{Q}\varphi$ if and only if $\Gamma\vdash_{Q}\varphi\dotminus 2^{-n}$ for all $n<\omega$.  Moreover, if $\mathcal{L}$ is a continuous signature with a metric, then $\Gamma\vDash_{QC}\varphi$ if and only if $\Gamma\vdash_{Q}\varphi\dotminus 2^{-n}$ for all $n<\omega$.
\end{cor} 

\begin{proof} Right to left is by soundness (Theorem \ref{prop:sound}).  For the implication from left to right, if $\Gamma\vDash\varphi$, then for every $n<\omega$, $\Gamma\cup\{2^{-n}\dotminus\varphi\}$ is not (completely) satisfiable and so inconsistent by Theorem \ref{thm:comp}; hence, $\Gamma\vdash\varphi\dotminus 2^{-n}$ for all $n<\omega$ (cf. proof of Lemma \ref{lem:max}). 
\end{proof}
We can, however, try to make the best of the previous result.  We first offer a definition.

\begin{defn} 
\label{defn:deg}
Let $\mathcal{L}$ be a continuous signature (possibly with a metric), let $\Gamma\subseteq \textup{For}(\mathcal{L})$, and let $\varphi$ be a formula.  
\mbox{}
\begin{itemize}
\item[(i)] We define the \emph{degree of truth of} $\varphi$ \emph{with respect to} $\Gamma$,  $\varphi_{\Gamma}^{\footnotesize{\circ}}$, by setting
\begin{eqnarray*}
 \varphi_{\Gamma}^{\footnotesize{\circ}}&\equals&\sup\{\mathfrak{M}(\varphi,\sigma):(\mathfrak{M},\sigma)\vDash\Gamma\}. 
 \end{eqnarray*}
\item[(ii)] We define the \emph{degree of provability of} $\varphi$ \emph{with respect to} $\Gamma$,  $\varphi_{\Gamma}^{\footnotesize{\circledcirc}}$, by setting
\begin{eqnarray*}
 \varphi_{\Gamma}^{\tiny{\circledcirc}}&\equals&\inf\{p\in\mathbb{D}:\Gamma\vdash\varphi\dotminus p\}.
 \end{eqnarray*}
 \end{itemize}
\end{defn}

We then have the following result, commonly called \emph{Pavelka-style completeness}.
\begin{cor}
Let $\mathcal{L}$ be a continuous signature, and let $\Gamma\subseteq \textup{For}(\mathcal{L})$.  Then for every formula $\varphi$,  the degree of truth of $\varphi$ with respect to $\Gamma$ equals the degree of provability of $\varphi$ with respect to $\Gamma$.  In other words,
\begin{eqnarray*}
 \varphi_{\Gamma}^{\footnotesize{\circ}}=\varphi_{\Gamma}^{\footnotesize{\circledcirc}}.
  \end{eqnarray*}
\end{cor}

\begin{proof}
We consider two cases:
\mbox{}
\begin{itemize}
\item[(i)] We first show $\varphi_{\Gamma}^{\footnotesize{\circ}}\leq\varphi_{\Gamma}^{\footnotesize{\circledcirc}}$. This follows from soundness (Theorem \ref{prop:sound}).  To see this, observe that for every $p\in\mathbb{D}$ such that $\Gamma\vdash\varphi\dotminus p$, by soundness $\Gamma\vDash\varphi\dotminus p$, so for any continuous $\mathcal{L}$(-pre)-structure $\mathfrak{M}$ and $\mathfrak{M}$-assignment $\sigma$ such that  $(\mathfrak{M},\sigma)\vDash\Gamma$, $\mathfrak{M}(\varphi,\sigma)\leq p$;  thus, $\varphi_{\Gamma}^{\footnotesize{\circ}}\leq p$.  It follows that $\varphi_{\Gamma}^{\footnotesize{\circ}}\leq \varphi_{\Gamma}^{\footnotesize{\circledcirc}}$.
\item[(ii)] We now show $\varphi_{\Gamma}^{\footnotesize{\circ}}\geq\varphi_{\Gamma}^{\footnotesize{\circledcirc}}$.  It suffices to show that for each $p\in\mathbb{D}$ such that $p<\varphi_{\Gamma}^{\footnotesize{\circledcirc}}$, there is a continuous $\mathcal{L}$(-pre)-structure $\mathfrak{M}$ and $\mathfrak{M}$-assignment $\sigma$ such that  $(\mathfrak{M},\sigma)\vDash\Gamma$ and $p\leq\mathfrak{M}(\varphi,\sigma)$.  Let $p\in\mathbb{D}$, and suppose $p<\varphi_{\Gamma}^{\footnotesize{\circledcirc}}$. Then $\Gamma\nvdash\varphi\dotminus p$, so $\Gamma\cup\{p\dotminus\varphi\}$ is consistent (cf. proof of Lemma \ref{lem:max}), whence by Theorem \ref{thm:comp}, $\Gamma\cup\{p\dotminus\varphi\}$ is (completely) satisfiable.  Hence, there is a continuous $\mathcal{L}$(-pre)-structure $\mathfrak{M}$ and $\mathfrak{M}$-assignment $\sigma$ such that  $(\mathfrak{M},\sigma)\vDash\Gamma$ and $\mathfrak{M}(p\dotminus\varphi,\sigma)=0$, so  $p\leq\mathfrak{M}(\varphi,\sigma)$.

\end{itemize}
\end{proof}

\begin{defn} Let $\mathcal{L}$ be a continuous signature with a metric.
\mbox{}
\begin{itemize}
\item[(i)] We call $T$ a \emph{theory} if $T$ is a set of formulae in $\mathcal{L}$ without free variables, i.e., a set of sentences.
\item[(ii)] We call a theory $T$ \emph{complete} if there is a continuous $\mathcal{L}$-structure $\mathfrak{M}$   (and an $\mathfrak{M}$-assignment $\sigma$) such that $T=\{\varphi:(\mathfrak{M},\sigma)\vDash_{QC}\varphi\}$.  Otherwise, we call $T$ an \emph{incomplete theory}.
\end{itemize}
\end{defn} 

\begin{defn} Let $T$ be a theory.  
\mbox{}
\begin{itemize}
\item[(i)] If $T$ is  complete, we say  $T$ is \emph{decidable} if for every sentence $\varphi$, the value $\varphi_{T}^{\footnotesize{\circ}}$ is a recursive real and uniformly computable from $\varphi$.
\item[(ii)] If $T$ is incomplete, we say $T$ is \emph{decidable} if for every sentence
$\varphi$ the real number $\varphi_{T}^{\footnotesize{\circ}}$ is uniformly recursive from $\varphi$.
\end{itemize}
\end{defn}

\begin{cor}
  Every complete theory with a recursive or recursively enumerable axiomatization is decidable.
\end{cor}

As with classical first-order logic, the following corollary could be obtained more directly by way of an ultraproduct construction (see \cite{BenYaacov-Usvyatsov:CFO}).  We nevertheless include it for the sake of completeness.
\begin{cor}[Compactness]
Let $\mathcal{L}$ be a continuous signature, and let $\Gamma\subseteq \textup{For}(\mathcal{L})$.  If every finite subset $\Gamma_{0}$ of $\Gamma$ is $($completely$)$ satisfiable, then $\Gamma$ is $($completely$)$ satisfiable.
\end{cor}

\providecommand{\bysame}{\leavevmode\hbox to3em{\hrulefill}\thinspace}
\providecommand{\MR}{\relax\ifhmode\unskip\space\fi MR }
\providecommand{\MRhref}[2]{%
  \href{http://www.ams.org/mathscinet-getitem?mr=#1}{#2}
}
\providecommand{\href}[2]{#2}


\end{document}